\documentclass[a4paper, 11pt]{amsart}

\usepackage{amsthm, amsmath, amssymb}
\usepackage{xcolor}
\usepackage{enumitem}

\usepackage[noadjust]{cite}
\usepackage[colorlinks, citecolor=red, urlcolor=blue, bookmarks=false, hypertexnames=true]{hyperref}

\theoremstyle{plain}
\newtheorem{theorem}{Theorem}[section]
\newtheorem*{theorem*}{Theorem \ref{thm:appl}}
\newtheorem{proposition}[theorem]{Proposition}
\newtheorem*{conjecture}{Conjecture}
\newtheorem{corollary}[theorem]{Corollary}
\newtheorem{lemma}[theorem]{Lemma}

\theoremstyle{definition}

\newtheorem{remark}[theorem]{Remark}

\newtheorem{innercustomgeneric}{\customgenericname}
\providecommand{\customgenericname}{}
\newcommand{\newcustomtheorem}[2]{%
	\newenvironment{#1}[1]
	{%
		\renewcommand\customgenericname{#2}%
		\renewcommand\theinnercustomgeneric{##1}%
		\innercustomgeneric
	}
	{\endinnercustomgeneric}
}

\newcustomtheorem{customthm}{Theorem}

\numberwithin{equation}{section}

\DeclareMathOperator{\trace}{trace}
\DeclareMathOperator{\grad}{grad}

\DeclareMathOperator{\id}{Id}

\title{Gap results for biharmonic submanifolds in spheres}
\author{\c Stefan Andronic}
\author{Simona Nistor}

\allowdisplaybreaks

\begin{document}
	\begin{abstract}
		In this paper we determine a larger gap of the mean curvature for a class of proper biharmonic submanifolds with parallel mean curvature vector field in Euclidean spheres. When the bounds of the gap are reached, we obtain splitting results of the submanifold.
	\end{abstract}
	
	\maketitle
	
	\section{Introduction}
	
	Biharmonic maps between Riemannian manifolds represent a natural generalization of the well-known harmonic maps. They are characterized by the vanishing of the bitension field
	\begin{equation}\label{eq:Tau2Immersions}
		\tau_2(\varphi) = - \Delta ^\varphi \tau (\varphi) - \trace R^N \left ( d\varphi (\cdot), \tau (\varphi) \right ) d\varphi (\cdot),
	\end{equation}
	see \cite{Jiang1986-1} and \cite{Jiang1986-2}, where $\tau(\varphi)$ is the tension field of the smooth map $\varphi : M^m \to N^n$, $\Delta^\varphi$ is the rough Laplacian acting on the sections of the pull-back bundle $\varphi^{-1} (TN)$, $R^N$ is the curvature vector field of the target manifold and the trace is done with respect to the domain metric. Since any harmonic map is biharmonic, we are interested in biharmonic maps which are not harmonic, called proper biharmonic. There is a very rich literature in the field of harmonic maps, but we mention only \cite{EellsLemaire1983} and \cite{EellsSampson1964}.
	
	In the particular case of isometric immersions, that is of submanifolds, the biharmonic equation $\tau_2 (\varphi) = 0$ splits in the tangent and normal parts and defines the biharmonic submanifolds. 
	
	In the late 80's, independently, B.Y. Chen introduced the notion of biharmonic submanifolds in Euclidean spaces $\mathbb R^n$. If we consider the case when the target manifold is an Euclidean space in \eqref{eq:Tau2Immersions} we find the definition given by Chen. Also, he formulated the famous conjecture (see \cite{Chen1991})
	\begin{conjecture}
		Any biharmonic submanifold in $\mathbb R^n$ is minimal.
	\end{conjecture} 
	
	Even though there are partial results for this conjecture, see, for example, \cite{AkutagawaMaeta2013}, \cite{Dimitric1992}, \cite{FuYangZhang2022}, \cite{HasanisVlachos1995}, \cite{MontaldoOniciucRatto2016}, it is not proven in its full generality.
	
	Since in negatively curved ambient spaces there are mostly non-existence results, see, for example, \cite{CaddeoMontaldoOniciuc2002}, \cite{GuanLiVrancken2021}, \cite{Jiang1986-2}, \cite{Maeta2014-1}, \cite{Maeta2014-2}, \cite{Oniciuc2002}, and one exception, see \cite{OuTang2012}, we will consider only the case of positively curved ambient spaces, in particular that of Euclidean spheres.
	
	Even if biharmonic submanifolds are somehow rigid, they still enjoy several interesting features. For example, an important result of biharmonic submanifolds states that a proper biharmonic submanifold with constant mean curvature ($CMC$) in a unit Euclidean sphere $\mathbb S^n$ has the mean curvature bounded from above by $1$, and it reaches the bound if and only if the submanifold lies minimally in the small hypersphere of radius $1 / \sqrt 2$ in $\mathbb S^n$ (see \cite{BalmusMontaldoOniciuc2008} and \cite{OniciucPHD}). However it is not well understood if the range of the mean curvature $|H|$ fills in the entire interval $(0, 1]$. When the mean curvature vector field $H$ is parallel in the normal bundle ($PMC$), it was proved in \cite{BalmusOniciuc2012} that $|H| \in \left ( 0, (m-2) / m \right ] \cup \{1\}$. 
	
	For current status on biharmonic submanifolds we refer to \cite{FetcuOniciuc2022}, \cite{Oniciuc2012HabilitationThesis}, \cite{ChenOu2020}.
	
	In this paper we aim to further enlarge the gap for the range of the mean curvature of $PMC$ biharmonic submanifolds in spheres, imposing additional properties to the submanifold. More precisely, we consider the submanifold $M$ in $\mathbb S^n$ to be the extrinsic product, with respect to some splitting of $\mathbb R^{n+1}$, of two submanifolds.  
	
	The following theorem summarizes the main results of this paper.
	
	\begin{customthm}{\ref{th:SummarizeTheorem}}
		Let $\varphi_1 : M_1 ^{m_1} \to \mathbb S^{n_1} (r_1)$, $m_1 \geq 2$, and $\varphi_2 : M_2^{m_2} \to \mathbb S^{n_2} (r_2)$, $m_2 \geq 2$, be two non-minimal and $PMC$ submanifolds such that $M^m = M_1^{m_1} \times M_2^{m_2}$ is a proper biharmonic in $\mathbb S^n$, $r_1 ^2 + r_2^2 = 1$ and $n_1 + n_2 = n - 1$.
		\begin{enumerate}
			\item If $m_1 > 2$ or $m_2 > 2$, then
			$$
			\left | H^{\iota \circ \varphi} \right | \in \left ( 0, \frac {m-4} m \right ] \cup \{ 1 \}.
			$$
			\item If $m_1 = m_2 = 2$, then $\varphi_1$ and $\varphi_2$ are pseudo-umbilical, that is
			$$
			\left | H^{\iota \circ \varphi} \right | = 1.
			$$
		\end{enumerate}
	\end{customthm}
	When the mean curvature reaches the value $(m - 4) / m$, the submanifold $M$ splits further extrinsically as described in Theorem \ref{th:MainTheorem}. After the new splitting, the submanifold $M$ falls in a known class of examples of proper biharmonic submanifolds in spheres. When the mean curvature reaches the upper bound $1$, we have a codimension reduction result given in Theorem \ref{th:CharacterizationNormHIs1}. 
	
	We note that, using a different technique, the gap $((m-2)/m, 1)$ was also enlarged in \cite{Nistor2023} for the case of hypersurfaces.
	
	We will see that the biharmonic equation for $M$ implies that $M_1$ and $M_2$ satisfy equations of type $\tau_2 (\varphi) = \lambda \tau (\varphi)$, for some real constant $\lambda$. This type of submanifolds are called $\lambda$-biharmonic submanifolds and we first study them. We find an upper bound for the range of $|H|$ for $CMC$ $\lambda$-biharmonic submanifolds in Euclidean spheres (Theorem \ref{th:CMCLambdaBiharmonicImmersion}) and, if we strengthen the hypothesis of $CMC$, we obtain a gap for the range of $|H|$ for $PMC$ $\lambda$-biharmonic submanifolds (Theorem \ref{th:HForPMCLambdaBiharmonicImmersions} and Corollary \ref{th:HForPMCLambdaBihamonicHypersurfaces}). When the lower bound of the gap is reached we have a splitting result and when the upper bound is reached the submanifold is pseudo-umbilical.
		
	\section{Preliminaries}
	
	We recall here the fundamental equations for submanifolds, for example see \cite{ChenBook1973} and \cite{DajczerTojeiroBook2019}. Let $\varphi : M^m \to N^n$ be an isometric immersion, i.e. we consider $M$ a submanifold in a Riemannian manifold $N$. All Riemannian metrics are denoted by the same symbol $\langle \cdot, \cdot \rangle$. In order to simplify the notations, we will use the fact that any immersion is locally an embedding and we will identify locally $M^m$ with its image through $\varphi$. Also, a vector field $X$ tangent to $M$ will be identified with $d \varphi (X)$. We have
	\begin{itemize}
		\item the Gauss equation
		\begin{equation}\label{eq:GaussEquation}
			\langle R^N (X, Y) Z, W \rangle = \langle R(X, Y) Z, W \rangle - \langle B(X, W), B(Y, Z) \rangle + \langle B(Y, W), B(X, Z) \rangle;
		\end{equation}
		\item the Codazzi equation
		\begin{equation}\label{eq:CodazziEquation}
			\left( \nabla A_\eta \right) (X, Y) - \left( \nabla A_\eta \right) (Y, X) = A_{\nabla_X ^\perp \eta} Y - A_{\nabla _Y ^\perp \eta} X - \left( R^N (X, Y) \eta \right) ^\top;
		\end{equation}
		\item the Ricci equation
		\begin{equation}\label{eq:RicciEquation}
			\left( R^N (X, Y) \eta \right)^\perp = R^\perp (X, Y) \eta + B(A_\eta X, Y) - B(X, A_\eta Y),
		\end{equation}
	\end{itemize}
	where $B$ and $A_\eta$ denote the second fundamental form and the shape operator corresponding to the section $\eta$ of the normal bundle.
	
	According to our conventions, the minimal submanifolds are also $CMC$ and $PMC$, i.e. $|H|$ is constant and $\nabla ^\perp H = 0$, respectively. Moreover, pseudo-umbilical submanifolds, i.e. submanifolds with $A_H = |H|^2 \id$, include the minimal ones, where $\id$ is the identity operator.
	
	There are two large classes of examples of biharmonic submanifolds in Euclidean spheres. The first class is given by 
	\begin{proposition}[\cite{CaddeoMontaldoOniciuc2002}]
		Let $M^m$ be a minimal submanifold in $\mathbb S^n (a) \times \{ b\}$, where $a^2 + b^2 = 1$, $0 < a < 1$. Then $M$ is a proper biharmonic submanifold in $\mathbb S^{n+1}$ if and only if $a = 1/\sqrt 2$ and $b = \pm 1/\sqrt 2$.
	\end{proposition}
	In this case, the submanifold $M$ is $PMC$, pseudo-umbilical, $|H| = 1$ and $A_H$ is parallel.
	
	The second class of examples is represented by
	\begin{proposition}[\cite{CaddeoMontaldoOniciuc2002}]
		Let $M_1 ^{m_1}$, $M_2^{m_2}$ be two minimal submanifolds in $\mathbb S^{n_1} (r_1)$ and $\mathbb S^{n_2} (r_2)$, respectively, where $r_1^2 + r_2^2 = 1$ and $n_1 + n_2 = n - 1$. The manifold $M^m = M_1^{m_1} \times M_2^{m_2}$ is a proper biharmonic submanifold in $\mathbb S^{n}$ if and only if $r_1 = r_2 = 1/\sqrt 2$ and $m_1 \neq m_2$.
	\end{proposition}
	In this case, the submanifold $M$ is $PMC$, the shape operator $A_H$ is parallel with two constant distinct principal curvatures, and $|H| = |m_2 - m_1| / m$.
	
	There are also examples of proper biharmonic submanifolds which do not belong to the above two classes. However, with some additional hypotheses, a biharmonic submanifold falls in one of the previous two classes.
	
	An interesting property of biharmonic submanifolds is that their mean curvature is bounded from above. Moreover, if it reaches the upper bound, then the submanifold belongs to one of the two classes from above.	
	
	\begin{proposition}[\cite{OniciucPHD}] \label{th:IntervalH}
		Let $\varphi : M^m \to \mathbb S^n$ be a $CMC$ proper biharmonic immersion. Then $|H| \in (0, 1]$. Moreover, $|H| = 1$ if and only if $M$ is minimal in a hypersphere $\mathbb S^{n-1} \left( 1/\sqrt 2 \right) \subset \mathbb S^{n}$.
	\end{proposition}
	
	\begin{remark}
		When $|H|=1$, $\varphi (M) \subset \Pi$, where $\Pi$ is the hyperplane of $\mathbb R^{n+1}$ given by 
		$$
		\Pi : \langle \overline n, \overline r \rangle = \frac 1 2,
		$$
		where $\overline n = (H + \overline r) / 2$ is a constant vector field and $\overline r$ is the position vector field in $\mathbb R^{n+1}$. We note that $\overline r = \iota \circ \varphi$, where $\iota$ is the canonical inclusion of $\mathbb S^n$ in $\mathbb R^{n+1}$ (see \cite{ChenYano1972}).
	\end{remark}
	
	As a consequence of Proposition \ref{th:IntervalH} we have the following rigid result.
	
	\begin{proposition}[\cite{BalmusOniciuc2012}]\label{th:ProperBiharmonicSubmanifoldsInSpheresGeneral}
		Let $\varphi : {M^m} \to {\mathbb S^{n-1}(a)}$ be a submanifold in a small hypersphere $\mathbb S^{n-1} (a)$ in $\mathbb S^{n}$, where $a^2 + b^2 =1$, $a \in (0,1)$. Then $M$ is proper biharmonic in $\mathbb S^{n}$ if and only if either $a = 1/\sqrt 2$ and $M$ is minimal in $\mathbb S^{n-1}\left( 1/\sqrt 2 \right)$, or $ a > 1/\sqrt 2$ and $M$ is minimal in a small hypersphere $\mathbb S^{n-2}\left( 1/\sqrt 2 \right)$ of $\mathbb S^{n-1} (a)$. In both cases, $|H| = 1$ and in the latter we have a codimension reduction result, that is, $M$ lies in some totally geodesic hypersphere $\mathbb S^{n-2} (1 / \sqrt 2) \subset \mathbb S^{n-1} \subset \mathbb S^n$.
	\end{proposition}
	
	\begin{remark}
		Assume that $\mathbb S^{n-1} (a) = \mathbb S^{n-1} (a) \times \left \{ \sqrt {1 - a^2} \right \}$. Then, when $a > 1 / \sqrt 2$, $\varphi (M)$ also belongs to the hypersphere  
		$$
		\mathbb S^{n-1} (1 / \sqrt 2) = \mathbb S^n \cap \Pi,
		$$ 
		therefore $\varphi (M)$ belongs to the totally geodesic hypersphere of $\mathbb S^{n-1} (1 / \sqrt 2)$ given by 
		$$
		\mathbb S^{n-2} (1 / \sqrt 2) = \mathbb S^{n-1} (1 / \sqrt 2) \cap (\Pi_0 \cap \Pi),$$ 
		where 
		$$
		\Pi_0 : y^{n+1} = \sqrt {1 - a^2} \quad \text{and} \quad \Pi : \langle \overline n, \overline r \rangle = \frac 1 2.
		$$
		Moreover, if we consider the hyperplane $\Pi'$ of $\mathbb R^{n+1}$ determined by the perpendicular line to $\Pi_0 \cap \Pi$ containing the origin of $\mathbb R^{n+1}$, which is parallel to $\overline n$, and by $\Pi_0 \cap \Pi$ we obtain that
		$$
		\mathbb S^{n-2} \left ( \frac 1 {\sqrt 2} \right ) \subset \mathbb S^{n-1} = \mathbb S^n \cap \Pi'.
		$$
	\end{remark}
	
	The next two results give examples of hypotheses that force the biharmonic submanifold to fall in the second class.
	
	\begin{theorem}[\cite{BalmusOniciuc2012}]\label{th:IntervalHPMCProperBiharmonicImmersions}
		Let $\varphi : {M^m} \to {\mathbb S^n}$ be a $PMC$ proper biharmonic immersion. Suppose that $m>2$ and $|H| \in (0,1)$. Then $|H| \in \left( 0 , (m-2) / m \right]$ and $|H| = (m-2) / m$ if and only if locally $\varphi (M)$ is an open subset of a product 
		$$
		M_1 \times \mathbb S^1 \left( \frac 1 {\sqrt 2}\right) \subset \mathbb S^n,
		$$
		where $M_1$ is a minimal submanifold embedded in $\mathbb S^{n-2} \left(1 / \sqrt 2\right)$ and the splitting $\mathbb R^{n+1} = \mathbb R^{n-1} \oplus \mathbb R^2$ does not depend on the point of $M$. Moreover, if $M$ is complete, then the previous decomposition of $\varphi (M)$ is global, where $M_1$ is a complete minimal submanifold in $\mathbb S^{n-2} \left( 1 / {\sqrt 2}\right)$.
	\end{theorem}
	
	\begin{remark}
		We note that the above result was proved independently and using a different technique in \cite{WangWu2012}.
	\end{remark}
	
	\begin{proposition}[\cite{BalmusOniciuc2012}]\label{th:Theorem170}
		Let $\varphi : {M^m} \to {\mathbb S^n}$ be a $PMC$ proper biharmonic immersion with $\nabla A_H = 0$. Suppose that $m > 2$ and $|H| \in \left( 0, (m-2) / m \right)$. Then, $m> 4$ and, locally,
		\begin{equation*}
			\varphi (M) = M_1 ^{m_1} \times M_2^{m_2} \subset \mathbb S^{n_1} \left( 1/\sqrt 2\right) \times \mathbb S^{n_2} \left( 1/\sqrt 2 \right) \subset \mathbb S^n,
		\end{equation*}
		where $M_i$ is a minimal embedded submanifold in $\mathbb S^{n_i} \left( 1/\sqrt 2 \right)$, $m_i \geq 2$, $i \in \{ 1, 2\}$, $m_1 + m_2 = m$, $m_1 \neq m_2$, $n_1 + n_2 = n-1$. In this case $|H| = |m_1 - m_2| / m$. Moreover, if $M$ is complete, then the above decomposition of $\varphi (M)$ holds globally, where $M_i$ is a complete minimal submanifold in $\mathbb S^{n_i} \left( 1/\sqrt 2 \right)$, $i\in \{ 1,2\}$.
	\end{proposition}
	Now we consider the following situation. Let ${\varphi_1} : {M_1^{m_1}} \to {\mathbb S^{n_1} (r_1)}$, $ {\varphi_2} : M_2^{m_2} \to {\mathbb S^{n_2} (r_2)}$ be two immersions and consider $\iota : \mathbb S^{n_1} (r_1)\times \mathbb S^{n_2} (r_2) \to \mathbb S^n$ the canonical inclusion. We want to see when $M = M_1 \times M_2$ belongs to one of the two classes of examples mentioned at the beginning. That is, we want to see when there is another splitting of $\mathbb R^{n+1}$ such that, up to isometries of $\mathbb R^{n+1}$, $M = \tilde M_1 \times \tilde M_2$, with $\tilde M_1$ and $\tilde M_2$ minimal in spheres of radius $1 / \sqrt 2$. This objective can be also seen as a generalization of Proposition \ref{th:ProperBiharmonicSubmanifoldsInSpheresGeneral}. As we will see, this case is not so rigid as the one in Proposition \ref{th:ProperBiharmonicSubmanifoldsInSpheresGeneral}.
	
	We recall the following result.
	\begin{theorem}
		Let $\varphi_1 : M^{m_1}_1 \to \mathbb S^{n_1}(r_1)$ and $\varphi_2 : M^{m_2}_2 \to S^{n_2} (r_2)$ be two immersions, where $r_1^2 + r_2^2 = 1$ and $n_1 + n_2 = n - 1$. Then $M_1 \times M_2$ is a proper biharmonic submanifold in $\mathbb S^{n}$ if and only if
		\begin{equation}\label{eq:biharmonicitySubmanifoldTorusR}
			\left\{
			\begin{array}{l}
				|\tau(\varphi_1)|>0 \quad \text{or}\quad |\tau(\varphi_2)|>0 \quad \text{or} \quad \frac {m_1} {r_1^2} \neq \frac {m_2} {r_2^2}\\[10pt]
				\tau_2(\varphi_1) = 2 r_2^2 \left( \frac {m_1} {r_1^2} - \frac {m_2} {r_2^2} \right) \tau(\varphi_1)\\[10pt]
				\tau_2(\varphi_2) = 2 r_1^2 \left( \frac {m_2} {r_2^2} - \frac {m_1} {r_1^2} \right) \tau(\varphi_2)\\[10pt]
				\frac {|\tau(\varphi_1)|^2} {r_1^2} - \frac {|\tau(\varphi_2)|^2} {r_2^2} + \left ( r_2^2 - r_1^2 \right ) \left( \frac {m_2} {r_2^2} - \frac {m_1} {r_1^2} \right) ^2 = 0
			\end{array}
			\right..
		\end{equation}
	\end{theorem}
	We note that, for $r_1 = r_2 = 1 / \sqrt 2$ the system \eqref{eq:biharmonicitySubmanifoldTorusR} was given in \cite{CaddeoMontaldoOniciuc2002} and \cite{FetcuOniciuc2022}.
	
	\begin{corollary}\label{th:CorolComponentsAreCMC}
		Let $M^{m_1}_1$, $M^{m_2}_2$ be two submanifolds in $\mathbb S^{n_1}(r_1)$ and $S^{n_2} (r_2)$, respectively, where $r_1^2 + r_2^2 = 1$ and $n_1 + n_2 = n - 1$. If $M_1 \times M_2$ is a proper biharmonic submanifold in $\mathbb S^{n}$, then $M_1$ and $M_2$ are $CMC$ in $\mathbb S^{n_1} (r_1)$ and $\mathbb S^{n_2} (r_2)$, respectively.
	\end{corollary}
	\begin{corollary}[\cite{CaddeoMontaldoOniciuc2002}] \label{th:CorollaryBiharmonicityProductOfSpheres}
		Let $M_1 ^{m_1}$, $M_2^{m_2}$ be two minimal submanifolds in $\mathbb S^{n_1} (r_1)$ and $\mathbb S^{n_2} (r_2)$, respectively, where $r_1^2 + r_2^2 = 1$ and $n_1 + n_2 = n - 1$. The manifold $M_1 \times M_2$ is a proper biharmonic submanifold in $\mathbb S^{n}$ if and only if $r_1 = r_2 = 1 / {\sqrt 2}$ and $m_1 \neq m_2$.
	\end{corollary}
	\begin{corollary}
		Let $\varphi_1 : M^{m_1}_1 \to \mathbb S^{n_1}(r_1)$ and $\varphi_2 : M^{m_2}_2 \to S^{n_2} (r_2)$ be two immersions, where $r_1^2 + r_2^2 = 1$ and $n_1 + n_2 = n - 1$. If $M_1^{m_1}$ is proper biharmonic in $\mathbb S^{n_1} (r_1)$, then $M_1 \times M_2$ is proper biharmonic in $S^{n}$ if and only if 
		$$
		\tau_2 (\varphi_2) = 0, \quad \frac {m_1} {r_1^2} = \frac {m_2} {r_2^2} \quad \text{and} \quad \frac {|\tau(\varphi_1)|^2} {r_1^2} = \frac {|\tau(\varphi_2)|^2} {r_2^2}.
		$$
	\end{corollary}
	\begin{corollary}
		Let ${\varphi_1} : {M_1^{m_1}} \to {\mathbb S^{n_1} (r_1)}$ be an immersion, $r_1^2 + r_2^2 = 1$ and $n_1 + n_2 = n - 1$. Then $M_1 \times S^{n_2} (r_2)$ is a proper biharmonic submanifold in $\mathbb S^{n}$ if and only if 
		\begin{equation*}
			\left \{
			\begin{array}{l}
				|\tau(\varphi_1)| > 0 \quad \text{or} \quad \frac {m_1} {r_1^2} \neq  \frac {n_2} {r_2^2} \\[10pt]
				\tau_2(\varphi_1) = 2 r_2^2 \left( \frac {m_1} {r_1^2} - \frac {n_2} {r_2^2} \right) \tau(\varphi_1)\\[10pt]
				|\tau(\varphi_1)|^2 = r_1^2 \left ( r_1^2 - r_2^2 \right ) \left( \frac {n_2} {r_2^2} - \frac {m_1} {r_1^2} \right)^2	
			\end{array}
			\right..
		\end{equation*}
	\end{corollary}
	
	We note that, from equation \eqref{eq:biharmonicitySubmanifoldTorusR}, each immersion $\varphi_i$ satisfies an equation of type $\tau_2 (\varphi_i) = \lambda_i \tau (\varphi_i)$, $i \in \{ 1, 2 \}$, with $\lambda_1 \lambda_2 < 0$. Therefore, in order to answer to our question, we will study the so called $\lambda$-biharmonic submanifolds.
	
	\section{$\lambda$-biharmonic submanifolds}
	
	In this section we will present some characterization results of $\lambda$-biharmonic submanifolds which will be useful later. First, we define $\lambda$-biharmonic submanifolds. An immersion $\varphi : M^m \to N^n$ is said to be $\lambda$-biharmonic if $\tau_2 (\varphi) = \lambda \tau (\varphi)$, where $\lambda$ is a real constant.
	
	We know that the equation $\tau_2 (\varphi) = \lambda \tau (\varphi)$ has already appeared in the literature (see \cite{ArroyoGarayMencia1998}, \cite{ArvanitoyeorgosKaimakamisMagid2009}, \cite{Branding2020}, \cite{Chen1988}, \cite{ChenBook2015}, \cite{FerrandezLucas1991} \cite{Garay1994}), but it was not studied in our context of submanifolds in Euclidean spheres.
	
	\begin{theorem}\label{th:CharacterizationLambdaBiharmonicity}
		An immersion $\varphi : M^m \rightarrow \mathbb S^n (r)$ is $\lambda$-biharmonic if and only if
		\begin{equation}\label{eq:LambdaBiharmonicCharacterization}
			\left\{
			\begin{array}{l}
				\Delta^\perp H + \trace B(\cdot, A_H(\cdot)) + \left( \lambda - \frac m {r^2} \right) H = 0\\[5pt]
				4 \trace A_{\nabla^\perp_{(\cdot)} H} (\cdot) + m \grad |H|^2 = 0
			\end{array}
			\right..
		\end{equation}
	\end{theorem}

	\begin{proof}
		We consider $\{ X_i \}_{i \in \overline {1, m}}$ a local orthonormal frame field tangent to $M$ and we have
		$$
		\trace R^{\mathbb S^n(r)}(d\varphi (\cdot), H) d\varphi (\cdot) = - \frac m {r^2} H.
		$$
		As in the case of biharmonic immersions, we can prove that
		$$
		\Delta H = \Delta^\perp H + \trace B(\cdot, A_H(\cdot)) + 2\trace A_{\nabla^\perp _{(\cdot)} H} (\cdot) + \frac m 2 \grad |H|^2.
		$$
		Now, replacing in \eqref{eq:Tau2Immersions} we get \eqref{eq:LambdaBiharmonicCharacterization}.
	\end{proof}
	
	In the case of hypersurfaces, we can define the mean curvature function as being $f = \trace A / m$. Using this definition of the mean curvature, we can state the above result in the following way

	\begin{corollary}
		Let $M^m$ be a hypersurface in $\mathbb S^{m+1} (r)$, the sphere of radius $r$. Then $M$ is $\lambda$-biharmonic in $\mathbb S^{m+1} (r)$ if and only if
		\begin{equation}\label{eq:HypersurfacesLambdaBihamronicCharacterization}
			\left\{
			\begin{array}{l}
				\Delta f = \left( \frac m {r^2} - \lambda - |A|^2 \right) f\\[5pt]
				A(\grad f) = - \frac m 2 f \grad f
			\end{array}
			\right..
		\end{equation}
		Moreover, if $M$ is a non-minimal $CMC$ hypersurface, then $M$ is proper $\lambda$-biharmonic in $\mathbb S^{m+1} (r)$ if and only if
		\begin{equation}\label{eq:CMCHypersurfacesLambdaBihamronicCharacterization}
			\lambda < \frac m {r^2} \quad \text{and} \quad |A|^2 = \frac m {r^2} - \lambda.
		\end{equation}
	\end{corollary}

	\begin{proof}
		Without loss of generality, we can suppose that $M$ is oriented and $\eta$ is a unit globally defined section in the normal bundle $NM^m$.
		
		Since 
		\begin{align*}
			& H = f \eta, \quad \Delta^\perp H = (\Delta f) \eta, \quad \trace B(\cdot, A_H (\cdot)) = f |A|^2 \eta,\\
			& m \grad |H|^2 = 2m f \grad f, \quad \trace A_{\nabla^\perp_{(\cdot)} H} (\cdot) = A(\grad f),
		\end{align*}
		from \eqref{eq:LambdaBiharmonicCharacterization} the conclusion follows.
	\end{proof}

	In the special case of $PMC$ submanifolds, we obtain a simpler form of Theorem \ref{th:CharacterizationLambdaBiharmonicity}.
	
	\begin{theorem}
		Let $M^m$ be a $PMC$ submanifold in $\mathbb S^n (r)$. Then $M$ is $\lambda$-biharmonic in $\mathbb S^n (r)$ if and only if
		\begin{equation}\label{eq:PMC1LambdaBiharmonicCHaracterization}
			\trace B(\cdot, A_H(\cdot)) = \left( \frac m {r^2} - \lambda \right) H
		\end{equation}
		or, equivalently,
		\begin{equation}\label{eq:PMC2LambdaBiharmonicCHaracterization}
			\left\{
			\begin{array}{l}
				|A_H|^2 = \left( \frac m {r^2} - \lambda \right) |H|^2\\[5pt]
				\langle A_H, A_\xi \rangle = 0, \quad \forall \xi \in C(NM)\ \text{and}\ \xi \perp H
			\end{array}
			\right..
		\end{equation}
	\end{theorem}

	\begin{proof}
		Since $\nabla^\perp H = 0$, we obtain
		$$
		\grad |H|^2 = 0, \quad \Delta^\perp H = 0, \quad \trace A_{\nabla^\perp_{(\cdot)} H} (\cdot) = 0.
		$$
		Therefore, \eqref{eq:LambdaBiharmonicCharacterization} is equivalent to the following relation in the normal bundle of $M$
		\begin{equation}\label{eq:PMC1LambdaBiharmonicCHaracterizationInTheProof}
			\trace B(\cdot, A_H(\cdot)) = \left( \frac m {r^2} - \lambda \right) H.
		\end{equation} 
		Further, taking in \eqref{eq:PMC1LambdaBiharmonicCHaracterizationInTheProof} the inner product with $H$ and with $\xi \in C(NM)$, $\xi \perp H$, we obtain the conclusion.
	\end{proof}
	
	\smallskip
	
	\subsection{Examples of $\lambda$-biharmonic hypersurfaces}
	
	In the following we will present some simple examples of $\lambda$-biharmonic hypersurfaces in $\mathbb S^n (r)$ which will be useful for our study.  The first proposition provides umbilical examples of $\lambda$-biharmonic hypersurfaces in spheres.
	
	\begin{proposition}
		Let $\mathbb S^m (a) \times \{b\} \equiv \mathbb S^m (a)$ be a small hypersphere in $\mathbb S^{m+1} (r)$ and consider $\iota : \mathbb S^m (a) \rightarrow \mathbb S^{m+1} (r)$ the canonical inclusion, where $a^2+b^2=r^2$. Then $\mathbb S^m (a)$ is proper $\lambda$-biharmonic in $\mathbb S^{m+1}$ if and only if $\lambda < m / r^2$ and
		$$
		a = \sqrt { \frac {m r^2} {2m - \lambda r^2}}.
		$$
		Moreover, in this case $|H|^2 = 1/r^2 - \lambda / m$.
	\end{proposition}

	\begin{proof}
		It is not difficult to check that 
		$$
		\eta_p = \frac 1 r \left( \frac b a x^1, \ldots, \frac b a x^{m+1}, -a\right), \quad \forall p = \left( x^1, \ldots, x^{m+1}, b \right) \in \mathbb S^m (a),
		$$
		defines a unit section in the normal bundle $N\mathbb S^m (a)$ and 
		\begin{equation*}
			A = - \frac 1 r \frac b a \id.
		\end{equation*}
		Therefore, $f = - (1 / r) (b / a)$ is constant, $|A|^2 = (m / r^2) (b^2 / a^2)$ and \eqref{eq:CMCHypersurfacesLambdaBihamronicCharacterization} is equivalent to 
		$$
		\frac m {r^2} \frac {b^2} {a^2} = \frac m {r^2} - \lambda.
		$$
		In particular, we obtain $\lambda < m / r^2$. Further, the above relation can be written as
		\begin{equation*}
			\frac 1 {a^2} = \frac 2 {r^2} - \frac \lambda m
		\end{equation*}
		and from here the conclusion follows.
	\end{proof}
	
	The next result provides $CMC$ examples of $\lambda$-biharmonic hypersurfaces in spheres with two constant distinct principal curvatures ($A$ is still parallel).

	\begin{proposition}\label{th:ExampleProductSphereCircle}
		Consider the extrinsic product $\iota : \mathbb S^{m-1} (r_1) \times \mathbb S^1 (r_2) \rightarrow \mathbb S^{m+1} (r)$, where $\iota$ is the canonical inclusion, $r_1^2 + r_2^2 = r^2$, $r_1, r_2 \in (0, r)$ and $m \geq 2$. Then $\iota$ is proper $\lambda$-biharmonic if and only if one of the following three cases holds
		\begin{enumerate}
			\item $m > 2$, $\lambda =  (m - 2\sqrt{m - 1}) / {r^2}$ and
			\begin{align*}
				& r_1^2 = \frac {m - 1 - \sqrt {m - 1}} {m - 2} r^2,\quad r_2^2 = \frac { \sqrt {m - 1} - 1} {m - 2} r^2;
			\end{align*}
			\item $m \geq 2$, $\lambda < (m - 2\sqrt{m - 1}) / {r^2}$, $\lambda \neq 0$ and
			\begin{align*}
				& r_1^2 = \frac {3m - 2 - \lambda r^2 \mp \sqrt{\left( m - \lambda r^2 \right)^2 - 4(m - 1)}} {2 \left( 2m - \lambda r^2 \right)} r^2,\\
				& r_2^2 = \frac {m + 2 - \lambda r^2 \pm \sqrt{\left( m - \lambda r^2 \right)^2 - 4(m - 1)}} {2 \left( 2m - \lambda r^2 \right)} r^2 ;
			\end{align*}
			\item or $m > 2$, $\lambda = 0$ and $r_1^2 = r_2^2 = {r^2} / 2$.
		\end{enumerate}
	\end{proposition}
	\begin{proof}
		We can check that
		$$
		\eta_p = \frac 1 r \left( \frac {r_2} {r_1} x^1, \ldots, \frac {r_2} {r_1} x^m, -\frac {r_1} {r_2} x^{m + 1}, - \frac {r_1} {r_2} x^{m + 2} \right),
		$$
		for any $p = \left( x^1, \ldots, x^{m+2} \right) \in \mathbb S^{m-1} (r_1) \times \mathbb S^1 (r_2)$, is a unit section in the normal bundle and 
		\begin{equation*}
			AX = \frac 1 r \left( - \frac {r_2} {r_1} X_1, \frac {r_1} {r_2} X_2 \right), \quad \forall X = (X_1, X_2).
		\end{equation*}
		Therefore
		\begin{align*}
			f =& \frac 1 m (\trace A) = \frac 1 m \left( \sum_{i=1} ^{m-1} \langle AX_{1,i}, X_{1,i} \rangle + \langle AX_2, X_2 \rangle \right)\\
			  =& \frac 1 {mr} \left( - \frac {r_2} {r_1} (m - 1) + \frac {r_1} {r_2} \right)
		\end{align*}
		and 
		\begin{equation*}
			|A|^2 = \sum_{i=1} ^{m - 1} |AX_{1,i}|^2 + |AX_2|^2 = \frac 1 {r^2} \left( \frac {r_2^2} {r_1^2} (m - 1) + \frac {r_1^2} {r_2^2} \right).
		\end{equation*}
		Taking into account \eqref{eq:CMCHypersurfacesLambdaBihamronicCharacterization}, the inclusion $\iota$ is proper $\lambda$-biharmonic if and only if
		\begin{equation*}
			\left\{
			\begin{array}{l}
				- \frac {r_2} {r_1} (m - 1) + \frac {r_1} {r_2} \neq 0\\[5pt]
				\frac {r_2^2} {r_1^2} (m - 1) + \frac {r_1^2} {r_2^2} = m - \lambda r^2
			\end{array}
			\right..
		\end{equation*}
		Denoting $\alpha = r_2^2 / r_1^2$ the above system is equivalent to
		\begin{equation}\label{eq:SecondDegreeEquationExample2}
			\left\{
			\begin{array}{l}
				\alpha (m - 1) \neq 1\\[5pt]
				(m - 1) \alpha^2 - \left( m - \lambda r^2 \right) \alpha + 1 = 0
			\end{array}
			\right..
		\end{equation}
		The discriminant corresponding to the second equation from \eqref{eq:SecondDegreeEquationExample2} must be non-negative, that is
		$$
		\Delta = \left( m - \lambda r^2 \right)^2 - 4(m - 1) \geq 0.
		$$
		Case I: If $\Delta = 0$, we get that 
		\begin{equation*}
			\lambda = \frac {m - 2\sqrt{m - 1}}{r^2} \in \left(0, \frac m {r^2} \right)
		\end{equation*}
		and thus 
		\begin{equation*}
			\alpha = \frac {m - \lambda r^2} {2(m - 1)} = \frac 1 {\sqrt{m - 1}} > 0,
		\end{equation*}
		i.e. $r_1^2 = \sqrt{m - 1} r_2^2$. Since $r_1^2 + r_2^2 = r^2$, we obtain that 
		\begin{equation*}
			r_1^2 = \frac {\sqrt {m - 1}} {1 + \sqrt {m - 1}} r^2 \quad \text{and} \quad r_2^2 = \frac {1} {1 + \sqrt{m - 1}} r^2.
		\end{equation*}
		The first relation of \eqref{eq:SecondDegreeEquationExample2} is equivalent to $m \neq 2$.\\
		Case II: If $\Delta > 0$, we obtain that 
		\begin{equation*}
			\lambda < \frac {m - 2\sqrt{m - 1}} {r^2}.
		\end{equation*}
		In this case
		\begin{align*}
			& \alpha_1 = \frac {m - \lambda r^2 - \sqrt{\left( m - \lambda r^2 \right)^2 - 4(m - 1)}} {2(m - 1)} > 0,\\
			& \alpha_2 = \frac {m - \lambda r^2 + \sqrt{\left( m - \lambda r^2 \right)^2 - 4(m - 1)}} {2(m - 1)} > 0.
		\end{align*}
		Thus,
		\begin{equation*}
			r_1^2 = \frac {2(m - 1)} {3m - 2 - \lambda r^2 \pm \sqrt{\left(m - \lambda r^2 \right)^2 - 4(m - 1)}} r^2
		\end{equation*}
		and
		\begin{equation*}
			r_2^2 = \frac {m - \lambda r^2 \pm \sqrt{\left( m - \lambda r^2 \right)^2 - 4(m - 1)}} {3m - 2 - \lambda r^2 \pm \sqrt{\left( m - \lambda r^2 \right)^2 - 4(m - 1)}} r^2.
		\end{equation*}
		It is easy to check that $r_1, r_2 \in \left( 0, r \right)$.
		
		Now, we verify the first relation of \eqref{eq:SecondDegreeEquationExample2}. It is equivalent to 
		\begin{equation*}
			m - \lambda r^2 \pm \sqrt{\left( m - \lambda r^2 \right)^2 - 4(m - 1)} \neq 2.
		\end{equation*}
		Suppose we have equality, i.e. 
		\begin{equation}\label{eq:relation1ExampleTorus}
			m - \lambda r^2 - 2 \neq \mp \sqrt{\left( m - \lambda r^2 \right)^2 - 4(m - 1)}.
		\end{equation}
		Since $\lambda < (m - 2\sqrt{m - 1}) / r^2$, the left-hand side of \eqref{eq:relation1ExampleTorus} is positive, so the first relation of \eqref{eq:SecondDegreeEquationExample2} is satisfied for $\alpha = \alpha_2$. For $\alpha = \alpha_1$, we square the both sides of \eqref{eq:relation1ExampleTorus} and we get $4 \lambda r^2 \neq 0$. Therefore for $\lambda = 0$, $\alpha_1$ cannot satisfies the first relation of \eqref{eq:SecondDegreeEquationExample2}.
		
		In conclusion, for any $\lambda \neq 0$ both $\alpha_1$ and $\alpha_2$ satisfy the first relation of \eqref{eq:SecondDegreeEquationExample2}. If $\lambda = 0$, then only $\alpha_2$ verifies the first relation of \eqref{eq:SecondDegreeEquationExample2} and we obtain $r_1 = r_2 = r / \sqrt 2$.
	\end{proof}
	
	\smallskip
	
	\subsection{Properties of $PMC$ $\lambda$-biharmonic submanifolds}
	
	In the more general case of $CMC$ submanifolds, we have the following property concerning the range of the mean curvature for $\lambda$-biharmonic submanifolds.
	
	\begin{theorem}\label{th:CMCLambdaBiharmonicImmersion}
		Let $\varphi : M^m \rightarrow \mathbb S^n (r)$ be a $CMC$ proper $\lambda$-biharmonic immersion. Then $\lambda < m / r^2$ and $|H| \in \left( 0, \sqrt{1 / r^2 - \lambda / m} \right]$. Moreover, $|H| = \sqrt{1 / r^2 - \lambda / m}$ if and only if $M$ is minimal in the small hypersphere $\mathbb S^{n-1} \left( \sqrt {mr^2 / (2m - \lambda r^2)} \right) \subset \mathbb S^n (r)$.
	\end{theorem}
	\begin{proof}
		First, we take the inner product with $H$ in the first equation of \eqref{eq:LambdaBiharmonicCharacterization} and we get
		\begin{align*}
			 \langle \Delta^\perp H, H \rangle + \langle \trace B(\cdot, A_H(\cdot)), H \rangle - \left( \frac m {r^2} - \lambda \right) |H|^2 = 0,
		\end{align*}
		i.e.
		\begin{equation*}
			\langle \Delta^\perp H, H \rangle = \left( \frac m {r^2} - \lambda \right) |H|^2 - |A_H|^2.
		\end{equation*}
		Then, using the Weitzenb\" ock formula
		\begin{equation*}
			\frac 1 2 \Delta |H|^2 = \langle \Delta^\perp H, H \rangle - |\nabla^\perp H|^2
		\end{equation*}
		and the property of $M$ being $CMC$, we obtain 
		\begin{equation}\label{eq:Relation1CMClambdaBiharmonic}
			\left( \frac m {r^2} - \lambda \right) |H|^2 = |A_H|^2 + |\nabla^\perp H|^2.
		\end{equation}
		Thus $\lambda < m / r^2$.
		
		Let $p$ an arbitrary point of $M$ and consider $\{ e_i \}_{i \in \overline {1, m}}$ an orthonormal basis in $T_p M$ such that $A_H e_i = \mu_i e_i$, for any $i \in \overline {1, m}$.
		
		We know that $ \sum \mu_i = m |H|^2$ and $\sum \mu_i^2 = |A_H|^2$. Using \eqref{eq:Relation1CMClambdaBiharmonic} and Cauchy-Buniakovski-Schwarz inequality we achieve
		\begin{align*}
			\left( \frac m {r^2} - \lambda \right) |H|^2 = \sum _{i=1} ^m \mu_i^2 + |\nabla^\perp H|^2 \geq \frac {\left( \sum \limits_{i=1} ^m \mu_i \right)^2} m + |\nabla^\perp H|^2 = m|H|^4 + |\nabla^\perp H|^2.
		\end{align*}
		This implies
		\begin{align*}
			|H|^2 \left( \frac m {r^2} - \lambda - m|H|^2 \right) \geq |\nabla^\perp H|^2 \geq 0,
		\end{align*}
		on $M$. Since $|H|^2 \neq 0$, we obtain that 
		\begin{equation*}
			|H|^2 \leq \frac 1 {r^2} - \frac \lambda m.
		\end{equation*}
		If $|H| = \sqrt{1 / r^2 - \lambda / m}$, then $\nabla^\perp H = 0$, i.e. $M$ is $PMC$ in $\mathbb S^n (r)$, and we have equality in Cauchy-Buniakovski-Schwarz inequality, i.e. $M$ is pseudo-umbilical.
		
		Clearly, if $M$ is $PMC$ and pseudo-umbilical, then $|H| = \sqrt{1 / r^2 - \lambda / m}$.
		
		Further, according to a result from \cite{ChenYano1972}, $M$ lies as a minimal submanifold in a small hypersphere of $\mathbb S^n (r)$ of radius
		$$
		\frac r {\sqrt{1 + r^2 |H|^2}} = \sqrt {\frac {m r^2} {2m - \lambda r^2}}.
		$$
	\end{proof}
	
	\begin{remark}
		We note that, if $\lambda \geq m / r^2$ and $M$ is $\lambda$-biharmonic, then $M$ has to be minimal.
	\end{remark}
	
	\begin{remark}
		From the proof of the above result, we can see that, apparently, Theorem \ref{th:CMCLambdaBiharmonicImmersion} can be given in a stronger form. Indeed, let $\varphi : M^m \rightarrow \mathbb S^n (r)$ be a $CMC$ proper $\lambda$-biharmonic immersion. Then $|H| = \sqrt{1/r^2 - \lambda/m}$ if and only if there is a point $p \in M$ such that, at $p$, $M$ is $PMC$ and pseudo-umbilical.
	\end{remark}

	\begin{corollary}
		Let $\varphi : M^m \rightarrow \mathbb S^{m+1} (r)$ be a $CMC$ proper $\lambda$-biharmonic hypersurface. Then $\lambda < m / r^2$ and $|H| \in \left( 0, \sqrt{1 / r^2 - \lambda / m} \right]$. Moreover, $|H| = \sqrt{1 / r^2 - \lambda / m}$ if and only if $\varphi (M)$ is an open subset of the small hypersphere 
		$$
		\mathbb S^{n-1} \left( \sqrt { \frac {mr^2} {2m - \lambda r^2} } \right) \subset \mathbb S^n (r).
		$$
	\end{corollary}
	
	The next theorem classifies the proper $\lambda$-biharmonic submanifolds in spheres which lie in a small hypersphere of the target.
	
	\begin{theorem}\label{th:lambdaBiharmonicSubmanifoldsThatLieInHyperspheres}
		Let $\varphi : M^m \to \mathbb S^{n-1} (a) \times \{ b \} \equiv \mathbb S^{n-1} (a)$ be an immersion and $\iota : \mathbb S^{n-1} (a) \to \mathbb S^n (r)$ be the canonical inclusion, where $a^2 + b^2 = r^2$ and $a > 0$. Then $M$ is proper $\lambda$-biharmonic in $\mathbb S^n (r)$ if and only if either $a^2 = m r^2 / \left ( 2m - \lambda r^2\right )$ and $M$ is minimal in $\mathbb S^{n-1} \left ( \sqrt { m r^2 / \left ( 2m - \lambda r^2 \right )} \right )$, or $a^2 > m r^2 / \left ( 2m - \lambda r^2\right )$ and $M$ is minimal in a small hypersphere $\mathbb S^{n-2} \left ( \sqrt { m r^2 / \left ( 2m - \lambda r^2\right )} \right )$ of $\mathbb S^{n-1} (a)$. In both cases, 
		\begin{equation*}
			\left | H ^{\iota \circ \varphi} \right |^2 = \frac 1 {r^2} - \frac \lambda m
		\end{equation*}
		and in the latter we have a codimension reduction result, that is $M$ lies in some totally geodesic hypersphere $\mathbb S^{n-2} \left ( \sqrt { m r^2 / \left ( 2m - \lambda r^2\right )} \right ) \subset \mathbb S^{n-1} (r) \subset \mathbb S^n (r)$.
	\end{theorem}
	
	\begin{proof}
		If $\iota \circ \varphi$ is $\lambda$-biharmonic, then it is not difficult to check that
		\begin{equation}\label{eq:lambdaBiharmonicityTau2}
			\left \{
			\begin{array}{l}
				\tau_2 (\varphi) = 2m \left ( \frac 1 {a^2} + \frac \lambda {2m} -\frac 1 {r^2} \right ) \tau(\varphi)\\[5pt]
				|\tau (\varphi) |^2 = \frac {m^2} {r^2} \left ( 2 - \frac {r^2} {a^2} \right ) - m \lambda
			\end{array}
			\right ..
		\end{equation}
		Note that from these relations we get that $\varphi$ is $CMC$ and $\lambda'$-biharmonic, where $\lambda'$ is the coefficient of $\tau (\varphi)$ from the first equation.
		
		It is easy to see that $\tau (\iota \circ \varphi) = \tau (\varphi) - (m / r) (b / a) \eta$, where $\eta$ is a unit section in the normal bundle of $\mathbb S^{n-1} (a)$ in $\mathbb S^n (r)$. Thus
		\begin{equation*}
			\left |H ^{\iota \circ \varphi} \right |^2 = \frac {1} {r^2} - \frac \lambda m
		\end{equation*}
		and from the second equation of \eqref{eq:lambdaBiharmonicityTau2} we obtain that
		\begin{equation*}
			\frac 1 {a^2} \leq \frac 2 {r^2} - \frac \lambda m.
		\end{equation*}
		The last relation implies that $a^2 \geq mr^2 / \left ( 2m - \lambda r^2 \right )$.
		
		If $a^2 = mr^2 / \left ( 2m - \lambda r^2 \right )$, then $\left |H^\varphi \right |^2 = 0$, i.e. $M$ is minimal in $\mathbb S^{n-1} (a)$.
		
		If $a^2 > mr^2 / \left ( 2m - \lambda r^2 \right )$ then, since $M$ is $CMC$ and $\lambda'$-biharmonic, from Theorem \ref{th:CMCLambdaBiharmonicImmersion} we obtain that 
		\begin{equation*}
			\left |H^\varphi \right | \in \left ( 0, \sqrt {\frac 2 {r^2} - \frac 1 {a^2} - \frac \lambda m} \right ].
		\end{equation*}
		Moreover, using the second equation of \eqref{eq:lambdaBiharmonicityTau2} we get that $\left | H^\varphi \right |^2 = 2 / r^2 - 1 / a^2 - \lambda / m$, thus $M$ is minimal in $\mathbb S^{n-2} \left ( \sqrt { mr^2 / \left ( 2m - \lambda r^2 \right )} \right )$.
	\end{proof}
	
	When the submanifold is $PMC$ we can prove that the range of the mean curvature is a bit smaller. More precisely, we can prove that, for $PMC$ proper $\lambda$-biharmonic submanifolds there is a gap in the interval given in Theorem \ref{th:CMCLambdaBiharmonicImmersion} where $|H|$ belongs to, see Theorem \ref{th:HForPMCLambdaBiharmonicImmersions} and Corollary \ref{th:HForPMCLambdaBihamonicHypersurfaces}.
	
	From now on, we will assume that $m > 2$. The case $m = 2$ will be studied at the end of the paper.
	
	In order to obtain this, we define the symmetric $(1, 1)$-tensor field 
	$$
	\Phi = A_H - |H|^2 \id,
	$$ 
	where $\id$ is the identity on $C(TM)$, see \cite{AlencarDoCarmo1994}, \cite{AliasGarcia2010}. It can be shown that $\trace \Phi = 0$ and 
	$$
	|\Phi|^2 = |A_H|^2 - m |H|^4.
	$$
	
	\begin{lemma}\label{th:LemmaCurvature}
		Let $\varphi : M^m \rightarrow \mathbb S^n (r)$ be a submanifold such that $H \neq 0$ at any point of $M$. Then
		\begin{align}
			R(X, Y) Z =& \left( \frac 1 {r^2} + |H|^2 \right) \left( \langle Z, Y \rangle X - \langle Z, X \rangle Y \right) \label{eq:LemmaCurvature}\\
					   & + \frac 1 {|H|^2} \left( \langle \Phi Y, Z \rangle \Phi X - \langle \Phi X, Z \rangle \Phi Y \right) \notag \\
					   & + \left( \langle \Phi Y, Z \rangle X - \langle \Phi X, Z \rangle Y + \langle Y, Z \rangle \Phi X - \langle X, Z \rangle \Phi Y \right) \notag \\
					   & + \sum _{a=1} ^k \left( \langle A_{\eta_a} Y, Z \rangle A_{\eta_a} X - \langle A_{\eta_a} X, Z \rangle A_{\eta_a} Y \right), \notag 
		\end{align}
		for any $X, Y, Z \in C(TM)$, where $\{ H / |H|, \eta_a \}_{a \in \overline {1, k}}$, $k = n - m - 1$, is an orthonormal frame field in the normal bundle of $M$ in $\mathbb S^n (r)$.
	\end{lemma}
	\begin{proof}
		Equation \eqref{eq:LemmaCurvature} follows from the Gauss equation \eqref{eq:GaussEquation}.
		
		Indeed, the terms in the right-hand side of \eqref{eq:GaussEquation} can be written as
		\begin{equation*}
			 \langle B(X, W), B(Y, Z) \rangle = \left\langle \frac 1 {|H|^2} \langle A_H Y, Z \rangle A_H X + \sum _{a=1} ^k \langle A_{\eta_a} Y, Z \rangle A_{\eta_a} X, W \right\rangle
		\end{equation*}
		and
		\begin{equation*}
			\langle B(Y, W), B(X, Z) \rangle = \left\langle \frac 1 {|H|^2} \langle A_H X, Z \rangle A_H Y + \sum _{a=1} ^k \langle A_{\eta_a} X, Z \rangle A_{\eta_a} Y, W \right\rangle.
		\end{equation*}
		Because these relations hold for any $W \in C(TM)$, we get that
		\begin{align*}
			R^{\mathbb S^n (r)} (X, Y) Z =& R(X, Y) Z + \frac 1 {|H|^2} \left( \langle A_H X, Z \rangle A_H Y - \langle A_H Y, Z \rangle A_H X \right)\\
			   							  & + \sum _{a=1} ^k \left( \langle A_{\eta_a} X, Z \rangle A_{\eta_a} Y - \langle A_{\eta_a} Y, Z \rangle A_{\eta_a} X \right).
		\end{align*}
		Now, taking into account that $A_H = \Phi + |H|^2 \id$ and
		\begin{equation*}
			R^{\mathbb S^n (r)} (X, Y) Z = \frac 1 {r^2} \left( \langle Z, Y \rangle X - \langle Z, X \rangle Y \right),
		\end{equation*}
		the conclusion follows.
	\end{proof}
	\begin{lemma}
		Let $\varphi : M^m \rightarrow \mathbb S^n (r)$ be a non-minimal $PMC$ submanifold. Then $\nabla \Phi$ is symmetric and 
		\begin{equation}\label{eq:RelationNabla2}
			\trace \nabla^2 \Phi = \left( \frac m {r^2} + m |H|^2 - \frac {|\Phi|^2} {|H|^2} \right) \Phi X + m \Phi^2 X - |\Phi|^2 X - \sum _{a=1} ^k \langle \Phi, A_{\eta_a} \rangle A_{\eta_a}.
		\end{equation}
	\end{lemma}
	\begin{proof}
		Using Codazzi equation \eqref{eq:CodazziEquation} and $\nabla ^\perp H = 0$ we obtain that 
		$$
		(\nabla A_H) (X, Y) = (\nabla A_H) (Y, X),
		$$
		for any $X, Y \in C(TM)$, i.e. $\nabla A_H$ is symmetric. Since $\nabla \Phi = \nabla A_H$, we obtain that $\nabla \Phi$ is symmetric.
		
		It can be easily shown that
		$$
		\trace (\nabla \Phi) = m \grad |H|^2 = 0
		$$
		and therefore $\trace (\nabla A_H) = \trace (\nabla \Phi) = 0$.
		
		Recall the Ricci commutation formula
		\begin{equation}\label{eq:RicciCommutationFormula}
			\left( \nabla^2 \Phi \right) (X, Y, Z) - \left( \nabla^2 \Phi \right) (Y, X, Z) = R(X, Y) \Phi Z - \Phi (R(X, Y) Z),
		\end{equation}
		for any $X, Y, Z \in C(TM)$.
		
		We consider $\{ X_i \}_{i \in \overline {1, m}}$ a local orthonormal frame field tangent to $M$ and let $\{ H / |H|, \eta_a \}_{a \in \overline {1, k}}$, $k = n - m - 1$, be an orthonormal frame field in the normal bundle of $M$ in $\mathbb S^n (r)$. We have
		$$
		\trace A_{\eta_a} = \langle mH, \eta_a \rangle = 0, \quad \forall a \in \overline {1, k}.
		$$
		Using Lemma \ref{th:LemmaCurvature} and equation \eqref{eq:RicciCommutationFormula} we infer
		\begin{align*}
			\trace (\nabla^2 \Phi) (X) =& \sum _{i=1} ^m (\nabla^2 \Phi) (X, X_i, X_i) + m \left( \frac 1 {r^2} + |H|^2 \right) \Phi X - \frac {|\Phi|^2} {|H|^2} \Phi X - |\Phi|^2 X\\
									    & + m \Phi^2 X + \sum _{a=1} ^k ( A_{\eta_a}( \Phi( A_{\eta_a} X)) - \Phi( A_{\eta_a}( A_{\eta_a} X)) - \langle A_{\eta_a}, \Phi \rangle A_{\eta_a} X ).
		\end{align*}
		By straightforward computations we obtain 
		\begin{equation*}
			\sum _{i=1} ^m (\nabla^2 \Phi) (X, X_i, X_i) = \nabla _X (\trace \nabla \Phi) = 0
		\end{equation*}
		and from Ricci equation \eqref{eq:RicciEquation} it follows that $\Phi \circ A_{\eta_a} = A_{\eta_a} \circ \Phi$.
		
		Therefore, we obtain the conclusion.
	\end{proof}
	\begin{lemma}\label{th:LemmaRelationFromWeitzenbock}
		Let $\varphi : M^m \to \mathbb S^n (r)$ be a non-minimal submanifold. If $M$ is $PMC$ and $A_H \perp A_{\eta_a}$, for any $a \in \overline {1, k}$, then
		\begin{equation*}
			- \frac 1 2 \Delta |\Phi|^2 =  \left( \frac m {r^2} + m |H|^2 - \frac {|\Phi|^2} {|H|^2} \right) |\Phi|^2 + m \trace \Phi^3 + |\nabla \Phi|^2.
		\end{equation*}
	\end{lemma}
	\begin{proof}
		It is easy to see that
		$$
		\langle \Phi, A_{\eta_a} \rangle = 0.
		$$
		Thus \eqref{eq:RelationNabla2} becomes 
		$$
		(\trace \nabla^2 \Phi) (X) = \left( \frac m {r^2} + m |H|^2 - \frac {|\Phi|^2} {|H|^2} \right) \Phi X + m \Phi^2 X - |\Phi|^2 X.
		$$
		The Weitzenb\" ock formula states that
		$$
		- \frac 1 2 \Delta |\Phi|^2 =  \langle \Phi, \trace \nabla^2 \Phi \rangle + |\nabla \Phi|^2
		$$
		and using the fact that
		\begin{equation*}
			 \langle \Phi, \trace \nabla^2 \Phi \rangle =  \left( \frac m {r^2} + m |H|^2 - \frac {|\Phi|^2} {|H|^2} \right) |\Phi|^2 + m \trace \Phi^3,
		\end{equation*}
		we get the conclusion.
	\end{proof}
	
	Before giving the main result concerning the range of $|H|$ for $\lambda$-biharmonic submanifolds, we first recall Okumura Lemma \cite{Okumura1974}.
	
	\begin{lemma}
		Let $b_1, \ldots, b_m$ be real numbers such that $\sum \limits _{i=1} ^m b_i = 0$. Then
		\begin{equation*}
			- \frac {m - 2} {\sqrt{m (m - 1)}} \left( \sum _{i=1} ^m b_i^2 \right)^{\frac 3 2} \leq \sum _{i=1} ^m b_i^3 \leq \frac {m - 2} {\sqrt {m (m - 1)}} \left( \sum _{i=1} ^m b_i^2 \right) ^{\frac 3 2}.
		\end{equation*}
		Moreover, equality holds in the right-hand (respectively, left-hand) side if and only if $(m - 1)$ of the $b_i$'s are non-positive (respectively, non-negative) and equal.
	\end{lemma}
	
	\begin{theorem}\label{th:HForPMCLambdaBiharmonicImmersions}
		Let $\varphi : M^m \to \mathbb S^n (r)$ be a $PMC$ proper $\lambda$-biharmonic immersion. Suppose that $m > 2$ and $|H|^2 \in ( 0, 1 / r^2 - \lambda / m )$. Then 
		$$
		\lambda \leq \frac {m - 2 \sqrt{m-1}} {r^2}
		$$
		and
		\begin{equation*}
			|H|^2 \in \left\{
			\begin{array}{ll}
				\left(0, x_2 \right], & \text{if}\ \lambda \leq 0\\[5pt]
				[x_1, x_2], & \text{if}\ \lambda \in \left( 0, \frac {m - 2\sqrt{m-1}} {r^2} \right)\\[5pt]
				\left\{ x_2 \right\}, & \text{if}\ \lambda = \frac {m - 2\sqrt{m-1}} {r^2} 
			\end{array}
			\right. ,
		\end{equation*}
		where $x_1$ and $x_2$ are the positive solutions of the following second degree equation
		$$
		m^4 x^2 + m^2 \left( m \lambda - \frac {(m-2)^2} {r^2} \right) x + (m-1)\lambda^2 = 0.
		$$
		Moreover, we have
		\begin{enumerate}
			\item $|H|^2 = x_2$ if and only if locally $M = L \times I$, $\varphi = \varphi_1 \times \varphi_2$ is an extrinsic product, $\varphi_1 : L \to \mathbb S^{n-2} (r_1)$ is minimal and $\varphi_2 : I \to \mathbb S^1 (r_2)$ is (an open subset of) the circle parametrized by arc length, where $I = (-\varepsilon, \varepsilon)$ and the radii are given by
			\begin{align}
				& r_1^2 = \frac {3m - 2 - \lambda r^2 - \sqrt{(m - \lambda r^2)^2 - 4(m - 1)}} {2(2m - \lambda r^2)} r^2, \label{eq:Radius1}\\
				& r_2^2 = \frac {m + 2 - \lambda r^2 + \sqrt{(m - \lambda r^2)^2 - 4(m - 1)}} {2(2m - \lambda r^2)} r^2 \label{eq:Radius2}.
			\end{align}
			If $M$ is complete and simply connected, the above decompositions of $M$ and $\varphi$ hold globally, $M = L \times \mathbb R$ and $\varphi_2 (\mathbb R) = \mathbb S^1 (r_2)$.
			
			\item if $\lambda > 0$, then $|H|^2 = x_1$ if and only if locally $M = L \times I$, $\varphi = \varphi_1 \times \varphi_2$ is an extrinsic product, $\varphi_1 : L \to \mathbb S^{n-2} (r_1)$ is minimal and $\varphi_2 : I \to \mathbb S^1 (r_2)$ is (an open subset of) the circle parametrized by arc length, where $I = (-\varepsilon, \varepsilon)$ and the radii are given by
			\begin{align}
				& r_1^2 = \frac {3m - 2 - \lambda r^2 + \sqrt{(m - \lambda r^2)^2 - 4(m - 1)}} {2(2m - \lambda r^2)} r^2, \label{eq:Radius3}\\
				& r_2^2 = \frac {m + 2 - \lambda r^2 - \sqrt{(m - \lambda r^2)^2 - 4(m - 1)}} {2(2m - \lambda r^2)} r^2. \label{eq:Radius4}
			\end{align}
			If $M$ is complete and simply connected, the above decompositions of $M$ and $\varphi$ hold globally, $M = L \times \mathbb R$ and $\varphi_2 (\mathbb R) = \mathbb S^1 (r_2)$.
		\end{enumerate}
		
	\end{theorem}
	\begin{proof}
		In order to get an upper bound of $|H|^2$, we look for an inequality involving $|H|^2$. Since $M$ is a $PMC$ submanifold and $\Phi$ is symmetric, we will use Lemma \ref{th:LemmaRelationFromWeitzenbock} and Okumura's Lemma to obtain the desired inequality.
		
		First, from \eqref{eq:PMC2LambdaBiharmonicCHaracterization} we get that
		$$
		|\Phi|^2 = |A_H|^2 - m|H|^4 = \left( \frac m {r^2} - \lambda \right) |H|^2 - m |H|^4 = \left( \frac m {r^2} - \lambda - m |H|^2\right) |H|^2,
		$$
		where, from Theorem \ref{th:CMCLambdaBiharmonicImmersion}, $\lambda < m / r^2$.
		
		Also, note that, as $\Phi$ is symmetric, we can apply Okumura's Lemma and obtain
		$$
		\trace \Phi^3 \geq - \frac {m-2} {\sqrt{m(m-1)}} |\Phi|^3,
		$$
		with equality if and only if $(m-1)$ of the eigenvalues of $\Phi$ are non-negative and equal.
		
		From Lemma \ref{th:LemmaRelationFromWeitzenbock} we get 
		\begin{align}
			 0 =& |\nabla \Phi|^2 + \left( \lambda + 2m |H|^2 \right) \left( \frac m {r^2} - \lambda - m |H|^2 \right) |H|^2 + m \trace \Phi ^3 \label{eq:MainInequality}\\
			\geq& \left(\lambda + 2m |H|^2 \right) \left( \frac m {r^2} - \lambda - m |H|^2 \right) |H|^2 - \frac {m(m-2)} {\sqrt{m(m-1)}} |H|^3 \left( \frac m {r^2} - \lambda - m|H|^2 \right)^{\frac 3 2} \notag\\
			   =& \left( \frac m {r^2} - \lambda - m|H|^2 \right) |H|^2 \left( \lambda + 2m |H|^2 - \frac {m(m-2)} {\sqrt{m(m-1)}} |H| \sqrt{\frac m {r^2} - \lambda - m |H|^2} \right).	\notag		   
		\end{align}
		Since $|H| < \sqrt{1 / r^2 - \lambda / m}$, we get that $m / r^2 - \lambda - m |H|^2 > 0$ and the last inequality is equivalent to
		\begin{equation}\label{eq:RelationBeforeSecondDegreeEquation}
			 \lambda + 2m |H|^2 \leq \frac {m(m - 2)} {\sqrt{m(m - 1)}} |H| \sqrt{\frac m {r^2} - \lambda - m |H|^2}.
		\end{equation}
		If $\lambda \geq 0$, then the left-hand side of inequality \eqref{eq:RelationBeforeSecondDegreeEquation} is positive. If $\lambda < 0$ and $|H|^2 \leq - \lambda / (2m)$, then the left-hand side of the inequality \eqref{eq:RelationBeforeSecondDegreeEquation} is non-positive, thus the inequality is automatically satisfied. Further, we will assume either $\lambda \geq 0$ or $\lambda < 0$ and $|H|^2 > -\lambda / (2m)$.
		
		Since $\lambda + 2m|H|^2 > 0$, by squaring, the inequation \eqref{eq:RelationBeforeSecondDegreeEquation} is equivalent to
		\begin{equation}\label{eq:SecondDegreeEquationInH}
			m^4 |H|^4 + m^2 \left( m \lambda - (m - 2)^2 \frac 1 {r^2} \right) |H|^2 + (m-1) \lambda^2 \leq 0.
		\end{equation}
		Let $x = |H|^2$, thus \eqref{eq:SecondDegreeEquationInH} becomes
		\begin{equation}\label{eq:SecondDegreeEquationInX}
			m^4 x^2 + m^2 \left( m \lambda - (m - 2)^2 \frac 1 {r^2} \right)x + (m - 1) \lambda^2 \leq 0.
		\end{equation}
		We compute now the discriminant of the second degree equation associated to \eqref{eq:SecondDegreeEquationInX}
		\begin{equation*}
			\Delta = \frac {m^4 (m - 2)^2} {r^4} \left( (\lambda r^2 - m)^4 - 4 (m - 1) \right).
		\end{equation*}
		Since $\Delta$ must be non-negative we have
		\begin{equation*}
			\lambda \leq \frac {m - 2 \sqrt{m - 1}} {r^2}
		\end{equation*}
		and we distinguish two cases.
		
		If $\Delta = 0$, then 
		\begin{align*}
			\lambda = \frac {m - 2 \sqrt{m - 1}} {r^2} \in \left( 0, \frac m {r^2} \right)
		\end{align*}
		and
		\begin{equation*}
			x_1 = x_2 = |H|^2 = \frac {2 - 2m + m\sqrt{m - 1}} {m^2 r^2}.
		\end{equation*}
		If $\Delta > 0$, then
		\begin{align*}
			\left( \lambda r^2 - m \right)^2 - 4 (m - 1) > 0 \Leftrightarrow \lambda < \frac {m - 2 \sqrt{m - 1}} {r^2}.
		\end{align*}
		The second degree equation associated to \eqref{eq:SecondDegreeEquationInX} has two real solutions
		\begin{equation*}
			x_1 = \frac {(m - 2)^2 - m \lambda r^2 - (m - 2) \sqrt{\left(\lambda r^2 - m \right)^2 - 4 (m - 1)}} {2 m^2 r^2}
		\end{equation*}
		and
		\begin{equation*}
			x_2 = \frac {(m - 2)^2 - m \lambda r^2 + (m - 2) \sqrt{\left(\lambda r^2 - m \right)^2 - 4 (m - 1)}} {2 m^2 r^2}.
		\end{equation*}
		We have to check whether
		$$
		x_1, x_2 \in \left( 0, \frac 1 {r^2} - \frac \lambda m \right)
		$$
		when $\lambda \geq 0$ and 
		$$
		x_1, x_2 \in \left ( - \frac \lambda {2m}, \frac 1 {r^2} - \frac \lambda m \right )
		$$
		when $\lambda < 0$.
		
		First, we check whether $x_1 > 0$, which is equivalent to 
		\begin{equation}\label{eq:RelationX1Positive}
			(m - 2)^2 - m\lambda r^2 > (m - 2) \sqrt{(\lambda r^2 - m)^2 - 4(m - 1)}.
		\end{equation}
		Since $\lambda < (m - 2 \sqrt{m - 1}) / r^2$, we have that the left-hand side of \eqref{eq:RelationX1Positive} is positive.
		
		Therefore, \eqref{eq:RelationX1Positive} is equivalent to
		\begin{equation*}
			 m^2 \lambda^2 r^4 > (m - 2)^2 \lambda^2 r^4,
		\end{equation*}
		which is true if $\lambda \neq 0$. Thus, $x_1$ and also $x_2$ are positive when $\lambda \neq 0$ and, if $\lambda = 0$, we have $0 = x_1 < x_2 = (m - 2)^2/(m^2 r^2)$.
		
		Next, one can check that
		\begin{equation*}
			 x_2 < \frac 1 {r^2} - \frac \lambda m
		\end{equation*}
		is true for any $\lambda < (m - 2 \sqrt{m - 1}) / r^2$.
		
		Since $x_1 < x_2$, we get that
		\begin{equation*}
			0 \leq x_1 < x_2 < \frac 1 {r^2} - \frac \lambda m,
		\end{equation*}
		for any $\lambda < (m - 2 \sqrt{m - 1}) / r^2$.
		
		If $\lambda > 0$, we conclude that inequality \eqref{eq:RelationBeforeSecondDegreeEquation} implies that 
		$$
		|H|^2 \in [x_1, x_2].
		$$
		If $\lambda = 0$,  we conclude that
		$$
		|H|^2 \in \left( 0, x_2 \right].
		$$
		If $\lambda < 0$ and $|H|^2 \in ( - \lambda / (2m), 1 / r^2 - \lambda / m )$ we have to check the relation between $x_1$, $x_2$ and $- \lambda / (2m)$. 
		
		If $x_1 \geq - \lambda / (2m)$, then $\lambda r^2 (\lambda r^2 - 2m) \leq 0$, which is false because $\lambda < 0$. Thus $x_1 < - \lambda / (2m)$. Relation $x_2 > - \lambda / (2m)$ is equivalent to
		\begin{equation*}
			m - 2 + \sqrt{\left( \lambda r^2 - m \right)^2 - 4(m - 1)} > 0,
		\end{equation*}
		which is always true. Thus, $x_2 > - \lambda / (2m)$ and
		\begin{equation*}
			|H|^2 \in \left( - \frac \lambda {2m}, x_2\right].
		\end{equation*}
		On the other hand, again for $\lambda < 0$, inequality \eqref{eq:RelationBeforeSecondDegreeEquation} is satisfied for $|H|^2 \leq - \lambda / (2m)$. Therefore, if $\lambda < 0$, then
		\begin{equation*}
			|H|^2 \in (0, x_2].
		\end{equation*}
		We note that $x_2$ is the maximum value of the squared mean curvature of $M^m$ in $\mathbb S^n (r)$, no matter the sign of $\lambda$, and $x_1$ is the minimum value of the squared mean curvature, when $\lambda > 0$.
		
		Now suppose that $|H|^2 = x_2$ or $|H|^2 = x_1$ and $\lambda > 0$. Thus we have equality in \eqref{eq:RelationBeforeSecondDegreeEquation} and all inequalities in \eqref{eq:MainInequality} become equalities, that is $\nabla \Phi = 0$ and we have equality in Okumura's Lemma.
		
		In order to obtain the second part of the theorem we will use the de Rham and Moore decomposition theorems. First, we find the principal curvatures of $A_H$ and their associated distributions. Indeed, we denote by $\sigma_i$, $i \in \overline {1, m}$, the eigenvalues of $\Phi$, which are real constants. Since we have equality in Okumura's Lemma, relabelling if necessarily, we can suppose that
		$$
		\sigma_1 = \sigma_2 = \ldots = \sigma_{m-1} = \sigma \geq 0.
		$$
		Further, since $\trace \Phi = 0$ we have 
		$$
		\sigma_m = - (m - 1) \sigma \leq 0
		$$
		and thus
		\begin{align*}
			|\Phi|^2 =& \sigma_1^2 + \cdots + \sigma_m^2 = m (m - 1) \sigma^2.
		\end{align*}
		We know that 
		$$
		|\Phi|^2 = \left( \frac m {r^2} - \lambda m |H|^2 \right) |H|^2 = \left( \frac m {r^2} - \lambda - m x_* \right) x_*,
		$$
		where $x_* = x_1$ or $x_* = x_2$. Therefore, when $|H|^2 = x_*$, we have
		\begin{equation*}
			\sigma^2_* = \frac {\left( \frac m {r^2} - \lambda - m x_* \right) x_*} {m (m - 1)}.
		\end{equation*}
		Using the fact that $x_*$ is a solution of the equation \eqref{eq:SecondDegreeEquationInX}, we obtain that
		\begin{equation*}
			\sigma_* = \frac {\sqrt {4m^2 x_* + r^2 \lambda^2}} {m^2 r}.
		\end{equation*}
		If we denote by $\theta_i$, $i \in \overline {1, m}$, the eigenvalues of $A_H$, we have 
		$$
		\theta_i = \sigma_i + |H|^2,
		$$
		and so, when $|H|^2 = x_*$, we get
		\begin{equation}\label{eq:ThetaFormula}
			\theta_* = \theta_1 = \ldots = \theta_{m - 1} = \sigma_* + |H|^2 = x_* + \frac {\sqrt {4m^2 x_* + r^2 \lambda^2}} {m^2 r}
		\end{equation}
		and
		\begin{equation}\label{eq:OmegaFormula}
			\omega_* = \theta_m = \sigma_m + |H|^2 = x_* - (m - 1) \frac {\sqrt {4m^2 x_* + r^2 \lambda^2}} {m^2 r}.
		\end{equation}
		For the sake of simplicity we redenote $x_* = x$, $\theta_* = \theta$ and $\omega_* = \omega$.
		
		We consider
		\begin{equation*}
			T_\theta = \left\{ X \in C(TM)\ : \ A_H X = \theta X \right\}, \quad \dim T_\theta = m - 1
		\end{equation*}
		and
		\begin{equation*}
			T_\omega = \left\{ X \in C(TM)\ : \ A_H X = \omega X \right\}, \quad \dim T_\omega = 1.
		\end{equation*}
		Since $\theta$ and $\omega$ are constants, it is well-known that $T_\theta$ and $T_\omega$ are orthogonal, involutive and parallel. In the following we apply, locally, an intrinsic and then extrinsic decomposition of $M$.
		
		Thus, we are in the hypotheses of the decomposition Theorem of de Rham and obtain that for any point $p_0 \in M$, there is an open neighbourhood $U \subset M$ of $p_0$ which is isometric with $L^{m-1} \times I$, where $I = (-\varepsilon, \varepsilon)$ corresponds to an integral curve of a unit vector field $\xi \in T_\omega$ on $U$ and $L$ corresponds to an integral submanifold of $T_\theta$ through $p_0$. Moreover, $L$ is totally geodesic in $U$ and the integral curves of $\xi$ are geodesics in $U$. We remark that $\xi$ is a parallel vector field on $U$.
		
		Next, we will identify $U$ with its image through $\varphi$ and we will apply the extrinsic decomposition theorem. For that, we first show that the integral curves of $\xi$, viewed as curves in $\mathbb R^{n+1}$, are circles, all of them lying in parallel $2$-planes. 
		
		Indeed, consider $\{ H / |H|, \eta_a \}_{a \in \overline {1, k}}$ an orthonormal frame field in the normal bundle of $M$ in $\mathbb S^n (r)$ and $\{ E_\ell \}_{\ell \in \overline {1, m - 1}}$ an orthonormal frame field in $T_\theta$ on $U$.
		
		We have
		\begin{align*}
			\trace B(\cdot, A_H(\cdot)) =& \sum _{\ell = 1} ^{m - 1} B(E_\ell, A_H E_\ell) + B(\xi, A_H \xi)\\
											=& \theta \sum _{\ell = 1} ^{m - 1} B(E_\ell, E_\ell) + \theta B(\xi, \xi) + (\omega - \theta) B(\xi, \xi)\\
											=& \theta m H + (\omega - \theta) B(\xi, \xi).
		\end{align*}
		Using \eqref{eq:PMC1LambdaBiharmonicCHaracterization} and the above relation we obtain
		\begin{equation}\label{eq:BXi}
			B(\xi, \xi) = \frac {\frac m {r^2} - \lambda - m \theta} {\omega - \theta} H.
		\end{equation}
		Since $A_H$ and $A_{\eta_a}$ commute, we have
		\begin{equation*}
			A_{\eta_a} \xi = 0, \quad \forall a \in \overline {1, k}.
		\end{equation*}
		We compute
		\begin{align*}
			\nabla ^{\mathbb S^n (r)} _{\xi} B(\xi, \xi) =& \frac {\frac m {r^2} - \lambda - m \theta} {\omega - \theta} \nabla ^{\mathbb S^n (r)} _{\xi} H \\
														 =& \frac {\frac m {r^2} - \lambda - m \theta} {\omega - \theta} \left( \nabla ^\perp _{\xi} H - A_H \xi \right)\\
														 =& \frac {\omega \left( \frac m {r^2} - \lambda - m \theta \right)} {\theta - \omega} \xi.
		\end{align*}
		Let $ \gamma : I \to U$ be an integral curve of $\xi$, $\gamma (0) \in L$ and set $F_1 = \dot \gamma = \xi \circ \gamma$. Since $\xi$ is parallel, $\gamma$ is a geodesic in $U$. We compute now $\nabla ^{\mathbb S^n (r)} _{\dot \gamma} F_1$.
		\begin{align*}
			\nabla ^{\mathbb S^n (r)} _{\dot \gamma} F_1 =& \nabla _{\dot \gamma} \dot \gamma + B(\dot \gamma, \dot \gamma) = B(\xi, \xi)\\
														 =& \frac {\frac m {r^2} - \lambda - m \theta} {\omega - \theta} H\\
														 =& \kappa F_2,
		\end{align*}
		where
		\begin{equation}\label{eq:kDefinition}
		\kappa = \left| \nabla ^{\mathbb S^n (r)} _{\dot \gamma} F_1 \right| = \left| \frac {\frac m {r^2} - \lambda - m \theta} {\omega - \theta} \right| |H|
		\end{equation}
		and 
		$$
		F_2 = \frac {\frac m {r^2} - \lambda - m \theta} {\kappa (\omega - \theta)} H.
		$$
		Next,
		\begin{align*}
			\nabla ^{\mathbb S^n (r)} _{\dot \gamma} F_2 =& \frac 1 \kappa \frac {\frac m {r^2} - \lambda - m \theta} {\omega - \theta} \nabla ^{\mathbb S^n (r)} _{\xi} H\\
														 =& \frac 1 \kappa \frac {\frac m {r^2} - \lambda - m \theta} {\omega - \theta} \left( \nabla ^\perp _{\xi} H - A_H \xi \right)\\
														 =& - \frac 1 \kappa \frac { \omega \left( \frac m {r^2} - \lambda - m \theta \right) } {\omega - \theta} \xi.
		\end{align*}
		Taking into account \eqref{eq:ThetaFormula}, \eqref{eq:OmegaFormula} and  \eqref{eq:kDefinition} one can prove that
		\begin{equation*}
			\kappa^2 = \frac {\omega \left( \frac m {r^2} - \lambda - m \theta \right)} {\omega - \theta}.
		\end{equation*}
		Thus 
		\begin{equation*}
			\nabla ^{\mathbb S^n (r)} _{\dot \gamma} F_2 = - \kappa F_1.
		\end{equation*}
		Further, let $\widetilde \gamma = \iota \circ \gamma : I \to \mathbb R^{n+1}$, where $\iota : \mathbb S^n (r) \to \mathbb R^{n+1}$ is the canonical inclusion. We set $\widetilde F_1 = \dot {\widetilde \gamma}$. Now, we compute 
		\begin{align*}
			\nabla ^{\mathbb R^{n+1}} _{\dot \gamma} \widetilde F_1 =& \nabla ^{\mathbb S^n (r)} _{\dot \gamma} F_1 - \frac 1 {r^2} \left \langle \dot \gamma, \widetilde F_1 \right \rangle \gamma = \kappa F_2 - \frac 1 {r^2} \gamma\\
				=& \alpha \widetilde F_2,
		\end{align*}
		where
		\begin{equation*}
			\alpha = \left| \nabla ^{\mathbb R^{n+1}} _{\dot \gamma} \widetilde F_1 \right| = \sqrt{ \kappa^2 + \frac 1 {r^2} }.
		\end{equation*}
		and
		\begin{equation*}
			\widetilde F_2 = \frac 1 {\sqrt{\kappa^2 + \frac 1 {r^2}}} \left( \kappa F_2 - \frac 1 {r^2} \gamma \right).
		\end{equation*}
		Next, we can see that
		\begin{align*}
			\nabla ^{\mathbb R^{n+1}} _{\dot \gamma} \widetilde F_2 =& \frac 1 { \alpha} \nabla ^{\mathbb R^{n+1}} _{\dot \gamma} \left( \kappa F_2 - \frac 1 {r^2} \gamma \right)\\
																=& \frac 1 { \alpha} \left( \kappa \nabla ^{\mathbb S^n (r)} _{\xi} F_2 - \frac \kappa {r^2} \langle \xi, F_2 \rangle \gamma - \frac 1 {r^2} \xi \right)\\
																=& - \frac 1 { \alpha} \left( \kappa^2 + \frac 1 {r^2} \right) \xi = - \alpha \widetilde F_1.
		\end{align*}
		Therefore, we obtain the Frenet equations for $\widetilde \gamma$ in $\mathbb R^{n+1}$,
		\begin{equation*}
			\nabla ^{\mathbb R^{n+1}} _{\dot \gamma} \widetilde F_1 = \alpha \widetilde F_2 \quad \text{and} \quad \nabla ^{\mathbb R^{n+1}} _{\dot \gamma} \widetilde F_2 = - \alpha \widetilde F_1,
		\end{equation*}
		thus $\gamma$ is a circle of radius 
		\begin{equation*}
			\frac 1 {\alpha} = \frac 1 {\sqrt {\kappa^2 + \frac 1 {r^2}}} = \frac r {\sqrt{1 + r^2 \kappa^2}}
		\end{equation*}
		in $\mathbb R^{n+1}$ which lies in the $2$-plane generated by the vectors $\widetilde F_1 (0)$ and $\widetilde F_2 (0)$.
		
		We note that, by straightforward computations we have, for $x = x_* = x_1$,
		\begin{equation*}
			\frac r {\sqrt{1 + r^2 \kappa^2}} = \frac {m + 2 - \lambda r^2 - \sqrt{(m - \lambda r^2)^2 - 4(m - 1)}} {2(2m - \lambda r^2)} r^2
		\end{equation*}
		and, for $x = x_* = x_2$, 
		\begin{equation*}
			\frac r {\sqrt{1 + r^2 \kappa^2}} = \frac {m + 2 - \lambda r^2 + \sqrt{(m - \lambda r^2)^2 - 4(m - 1)}} {2(2m - \lambda r^2)} r^2.
		\end{equation*}
		We note that the radii obtain above correspond to the radii obtained in Proposition \ref{th:ExampleProductSphereCircle}.
		
		We show that the vector fields 
		$$
		\frac { \frac m {r^2} - \lambda - m \theta} {\omega - \theta} H - \frac 1 {r^2}\, \overline r \quad \text{and} \quad \xi \in T_\omega
		$$ 
		are parallel in $\mathbb R^{n+1}$ along all curves of $L$. In order to prove this, we consider $X \in C (TL)$ and we will show that 
		$$
		\nabla ^{\mathbb R^{n+1}} _X \left( \frac {\frac m {r^2} - \lambda - m \theta} {\omega - \theta} H - \frac 1 {r^2} \overline r \right) = 0 \quad \text{and} \quad \nabla ^{\mathbb R^{n+1}} _X \xi = 0.
		$$
		It is not difficult to see that $B(X, \xi) = 0$ and so
		\begin{equation}\label{eq:RelationXiParallelAlongL}
			\nabla ^{\mathbb R^{n+1}} _{X} \xi = \nabla ^{\mathbb S^{n} (r)} _{X} \xi = \nabla _X \xi = 0.
		\end{equation}
		Next, by direct computation, we have
		\begin{align*}
			\nabla ^{\mathbb R^{n+1}} _{X} \left( \frac {\frac m {r^2} - \lambda - m \theta} {\omega - \theta} H - \frac 1 {r^2} \overline r \right) =& \nabla ^{\mathbb R^{n+1}} _{X} \left( B(\xi, \xi) - \frac 1 {r^2} \overline r \right)\\
							=& \nabla ^{\mathbb R^{n+1}} _X \nabla ^{\mathbb S^n(r)} _\xi \xi - \frac 1 {r^2} X\\
							=& \nabla ^{\mathbb S^n(r)} _X \nabla ^{\mathbb S^n(r)} _\xi \xi - \frac 1 {r^2} \left\langle X, \nabla ^{\mathbb S^n(r)} _\xi \xi \right\rangle \overline r - \frac 1 {r^2} X\\
							=& \nabla ^{\mathbb S^n(r)} _\xi \nabla ^{\mathbb S^n(r)} _X \xi + \nabla ^{\mathbb S^n(r)} _{[X, \xi]} \xi + R(X, \xi) \xi - \frac 1 {r^2} X. 
		\end{align*}
		Since $[X, \xi] = 0$ and using \eqref{eq:RelationXiParallelAlongL} we obtain that
		\begin{align*}
			\nabla ^{\mathbb R^{n+1}} _{X} \left( \frac {\frac m {r^2} - \lambda - m \theta} {\omega - \theta} H - \frac 1 {r^2} \overline r \right) =& R(X, \xi) \xi - \frac 1 {r^2} X \\
						=& \frac 1 {r^2} \langle \xi, \xi \rangle X - \frac 1 {r^2} \langle \xi, X \rangle \xi - \frac 1 {r^2} X \\
						=& 0.
		\end{align*}
		We conclude that the $2$-planes determined by the integral curves of $\xi$ have the same vector space, thus are parallel.
		
		This does not ensure the splitting of $\mathbb R^{n+1}$. But, if we can apply the extrinsic decomposition Theorem of Moore, then we will have $\mathbb R^{n + 1} = \mathbb R^{n_0} \oplus \mathbb R^{n - n_0 - d + 1} \oplus \mathbb R^d$ and, since the vector spaces of $\mathbb R^{n+1}$ generated by $\xi (p)$, for any $p \in M$, have dimension $2$, we will obtain $d = 2$.
	
		We consider the immersion
		\begin{equation*}
			\varphi : L \times I \to \mathbb S^n (r)
		\end{equation*}
		and
		\begin{equation*}
			\widetilde \varphi = \iota \circ \varphi : L \times I \to \mathbb R^{n+1}.
		\end{equation*}
		For any $X \in C\left( TL \right)$ and $Y \in C \left( TI \right)$ we have
		\begin{equation*}
			\nabla ^{\mathbb R^{n+1}} _X Y = \nabla ^{\mathbb S^n (r)} _X Y - \frac 1 {r^2} \langle X, Y \rangle \overline r = \nabla _X Y + B(X, Y) = B(X, Y).
		\end{equation*}
		On the other hand,
		\begin{equation*}
			\nabla ^{\mathbb R^{n+1}} _X Y = \nabla _X Y + \widetilde B(X, Y) = \widetilde B(X, Y),
		\end{equation*}
		where $\widetilde B$ denotes the second fundamental form of $\widetilde \varphi$.
		Thus 
		\begin{equation*}
			B(X, Y) = \widetilde B(X, Y).
		\end{equation*}
		Since 
		\begin{equation*}
			\left\langle \widetilde B(X, Y), \frac H {|H|} \right\rangle = \frac 1 {|H|} \langle B(X, Y), H \rangle = \frac 1 {|H|} \langle A_H X, Y \rangle = \frac \theta {|H|} \langle X, Y \rangle = 0
		\end{equation*}
		and
		\begin{equation*}
			\langle \widetilde B(X, Y), \eta_a \rangle = \langle B(X, Y), \eta_a \rangle = \langle  X, A_{\eta_a} Y \rangle = 0,\quad \forall a \in \overline {1, k},
		\end{equation*}
		we get 
		\begin{equation*}
			\widetilde B(X, Y) = 0, \quad \forall X \in C \left( TL \right)\ \text{and}\ \forall Y \in C(TI).
		\end{equation*}
		In these conditions we can apply the extrinsic decomposition Theorem of Moore obtaining
		\begin{equation*}
			\mathbb R^{n+1} = \mathbb R^{n_0} \oplus \mathbb R^{n-n_0-1} \oplus \mathbb R^2
		\end{equation*}
		and $ L \times I = U = \widetilde \varphi (U)$, where $L^{m-1} \subset \mathbb R^{n-n_0-1}$ and $I$, identified with its image, is (an open subset of) a circle of radius $1 / \alpha$.
		
		Now, suppose $n_0 = 1$, then
		\begin{equation*}
			\widetilde \varphi = \left( v_0, \widetilde \varphi_1, \widetilde \varphi_2 \right).
		\end{equation*}
		Since $U \subset \mathbb S^n (r)$, we obtain that $U$ lies in a small hypersphere $\mathbb S^{n-1} (a)$ in $\mathbb S^n(r)$ and from Theorem \ref{th:lambdaBiharmonicSubmanifoldsThatLieInHyperspheres} we obtain that 
		\begin{equation*}
			|H|^2 = \frac 1 {r^2} - \frac \lambda m,
		\end{equation*}
		which contradicts the hypotheses. Thus 
		\begin{equation*}
			\mathbb R^{n+1} = \mathbb R^{n-1} \oplus \mathbb R^2.
		\end{equation*}
		Note that the center of the image of $I$ is the origin of $\mathbb R^2$. Indeed, for $p = (q, \widetilde s)$ from the image of $L \times I$,
		$$
		|q|^2 + |\widetilde s|^2 = r^2,
		$$
		which implies that $|q|^2 = constant$ on $L$ and $|\widetilde s|^2 = constant$ on $I$, i.e. the center of the image of $I$, which is a circle, is the origin of $\mathbb R^2$.
		
		Because $L \subset \mathbb R^{n-1}$ and $|q| = constant$ for any $q \in L$, we have that 
		$$
		L \subset \mathbb S^{n-2} \left( \sqrt{r^2 - \frac 1 {\alpha^2}} \right) = \mathbb S^{n-2} \left( \frac {\kappa r^2} {\sqrt{1 + \kappa^2r^2}} \right) \subset \mathbb R^{n-1}.
		$$
		We show that $L$ is pseudo-umbilical and $PMC$ in $\mathbb S^n (r)$. For that we denote by $\varphi_1$ the inclusion of $L$ (through $\varphi$) in $\mathbb S^n (r)$ and, since $L$ is totally geodesic in $U = L \times I$, we have 
		\begin{align*}
			B^{\varphi_1} (E_\ell, E_\ell) =& \nabla ^{\mathbb S^n(r)} _{E_\ell} E_\ell - \nabla ^L _{E_\ell} E_\ell\\
									 =& B(E_\ell, E_\ell) + \nabla _{E_\ell} E_\ell - \nabla _{E_\ell} E_\ell = B(E_\ell, E_\ell),
		\end{align*}
		thus 
		\begin{equation*}\label{eq:FormulaTau1}
			(m - 1) H^{\varphi_1} = \sum _{i=1} ^{m-1} B(E_\ell, E_\ell) = m H - B(\xi, \xi).
		\end{equation*}
		and from \eqref{eq:BXi} we get
		\begin{equation*}
			H^{\varphi_1} = \frac {\frac m {r^2} - \lambda - m \omega} {(m - 1)(\theta - \omega)} H.
		\end{equation*}
		Let $X \in C(TL)$. We have 
		\begin{align*}
			\nabla ^{\mathbb S^n (r)} _X H^{\varphi_1} =& \frac {\frac m {r^2} - \lambda - m \omega} {(m-1) (\theta - \omega)} \nabla ^{\mathbb S^n (r)} _X H\\
													   =& \frac {\frac m {r^2} - \lambda - m \omega} {(m-1) (\theta - \omega)} \left( \nabla ^\perp _X H - A_H X \right)\\
													   =& - \frac {\left( \frac m {r^2} - \lambda - m \omega \right) \theta} {(m-1) (\theta - \omega)} X.
		\end{align*}
		On the other hand,
		\begin{equation*}
			\nabla ^{\mathbb S^n (r)} _X H^{\varphi_1} = \nabla ^{\perp_{\varphi_1}} _X H^{\varphi_1} - A^{\varphi_1} _{H^{\varphi_1}} X,
		\end{equation*}
		and thus
		\begin{equation*}
			\left\{
			\begin{array}{l}
				\nabla^{\perp_{\varphi_1}} _X H^{\varphi_1} = 0\\
				A^{\varphi_1} _{H^{\varphi_1}} X = \frac {\left( \frac m {r^2} - \lambda - m \omega \right) \theta} {(m-1) (\theta - \omega)} X
			\end{array}
			\right.,
		\end{equation*}
		i.e. $\varphi_1$ is pseudo-umbilical and $PMC$ in $\mathbb S^n (r)$.
		Since $L$ is pseudo-umbilical and $PMC$ in $\mathbb S^n (r)$, it is pseudo-umbilical and $PMC$ in $\mathbb R^{n+1}$. This is equivalent to $L$ being pseudo-umbilical and $PMC$ in $\mathbb R^{n-1}$, which lastly, is equivalent to $L$ being pseudo-umbilical and $PMC$ in $\mathbb S^{n-2} (r_1)$.
		
		Using a result from \cite{ChenYano1972}, either $L$ lies as a minimal submanifold in $\mathbb S^{n-2} (r_1)$, or it lies minimal in a hypersphere of $\mathbb S^{n-2} (r_1)$. If $L$ lied in a hypersphere of $\mathbb S^{n-2} (r_1)$, then $L \times I$ would lie in a hypersphere of $\mathbb S^n (r)$ thus, according to Theorem \ref{th:lambdaBiharmonicSubmanifoldsThatLieInHyperspheres}, $|H| = \sqrt {1 / r^2 - \lambda / m}$, which is contradiction.
		
		Therefore, $L$ lies as a minimal submanifold of $\mathbb S^{n-2} (r_1)$ and the local version of our result is proved.
		
		Moreover, we can see that the splitting of $\mathbb R^{n+1}$ does not depend on the point $p_0$. Therefore the projection of the whole $\varphi (M)$ on $\mathbb R^2$ is an open subset of $\mathbb S^1 \left( r / \sqrt{1 + \kappa^2 r^2}\right)$ or the entire circle.
		
		It is not difficult to check that, conversely, if $M = L \times I$, where $L$ is minimal in $\mathbb S^{n-2} (r_1)$, $I$ is identified with $\mathbb S^1 (r_2)$ and $r_1$ and $r_2$ are given by \eqref{eq:Radius1}, \eqref{eq:Radius2}, then $|H|^2 = x_2$. Similarly, if $r_1$ and $r_2$ are given by \eqref{eq:Radius3}, \eqref{eq:Radius4} and $\lambda > 0$, then $|H|^2 = x_1$.
		
		For the global version of the result, when $M$ is complete and simply connected, we simply mention that the decomposition Theorem of de Rham holds globally.
	\end{proof}
	
	\begin{corollary}\label{th:HForPMCLambdaBihamonicHypersurfaces}
		Let $\varphi : M^m \to \mathbb S^{m+1} (r)$ be a $CMC$ proper $\lambda$-biharmonic hypersurface with $m>2$. Then 
		\begin{equation*}
			|H|^2 \in \left\{
			\begin{array}{ll}
				\left(0, x_2 \right] \cup \left \{ \frac 1 {r^2} - \frac \lambda m \right \} , & \text{if}\ \lambda \leq 0\\[5pt]
				[x_1, x_2] \cup \left \{ \frac 1 {r^2} - \frac \lambda m \right \} , & \text{if}\ \lambda \in \left( 0, \frac {m - 2\sqrt{m-1}} {r^2} \right)\\[5pt]
				\left\{ x_2, \frac 1 {r^2} - \frac \lambda m \right \} , & \text{if}\ \lambda = \frac {m - 2\sqrt{m-1}} {r^2} 
			\end{array}
			\right. ,
		\end{equation*}
		Moreover, we have
		\begin{enumerate}
			\item $|H|^2 = 1/r^2 - \lambda / m$ if and only if $\varphi (M)$ is an open subset of the small hypersphere $\mathbb S^m ( r / \sqrt {2 - (\lambda / m) r^2 } )$. 
			\item $|H|^2 = x_2$ if and only if $\varphi (M)$ is an open subset of the extrinsic product $\mathbb S^{m-1} (r_1) \times \mathbb S^1 (r_2)$, where
			\begin{align*}
				& r_1^2 = \frac {3m - 2 - \lambda r^2 - \sqrt{(m - \lambda r^2)^2 - 4(m - 1)}} {2(2m - \lambda r^2)} r^2, \\
				& r_2^2 = \frac {m + 2 - \lambda r^2 + \sqrt{(m - \lambda r^2)^2 - 4(m - 1)}} {2(2m - \lambda r^2)} r^2 .
			\end{align*}
			\item  if $\lambda>0$, then $|H|^2 = x_1$ if and only if $\varphi (M)$ is an open subset of the extrinsic product $\mathbb S^{m-1} (r_1) \times \mathbb S^1 (r_2)$, where
			\begin{align*}
				& r_1^2 = \frac {3m - 2 - \lambda r^2 + \sqrt{(m - \lambda r^2)^2 - 4(m - 1)}} {2(2m - \lambda r^2)} r^2, \\
				& r_2^2 = \frac {m + 2 - \lambda r^2 - \sqrt{(m - \lambda r^2)^2 - 4(m - 1)}} {2(2m - \lambda r^2)} r^2 .
			\end{align*}
		\end{enumerate}
	\end{corollary}
	
	\section{Biharmonic product submanifolds}
	
	Now, we come back to the main objective of our paper, that is to find a larger gap of the mean curvature of certain $PMC$ biharmonic submanifolds, and consider two non-minimal immersions $\varphi_1 : M_1^{m_1} \to \mathbb S^{n_1} (r_1)$, $m_1 > 2$, and $\varphi_2 : M_2^{m_2} \to \mathbb S^{n_2} (r_2)$, $m_2 > 2$, such that $M^m = M_1^{m_1} \times M_2^{m_2}$ is proper biharmonic in $\mathbb S^n$, i.e. $\iota \circ \varphi = \iota \circ (\varphi_1 \times \varphi_2)$ is a proper biharmonic immersion, where $\iota$ is the canonical inclusion of $\mathbb S^{n_1} (r_1) \times \mathbb S^{n_2} (r_2)$ in $\mathbb S^n$ and $n_1 + n_2 = n - 1$.
	
	Recall from Corollary \ref{th:CorolComponentsAreCMC} that $M_1$ and $M_2$ must be $CMC$. Also we know, from equation \eqref{eq:biharmonicitySubmanifoldTorusR}, that $\lambda_1 = 2 r_2^2 (m_1 / r_1^2 - m_2 / r_2^2 )$ and $\lambda_2 = 2 r_1^2 ( m_2 / r_2^2 - m_1 / r_1^2 )$.
	
	According to Theorem \ref{th:CMCLambdaBiharmonicImmersion}, the radii $r_1$ and $r_2$ must obey the following constraints
	\begin{equation*}
		\frac {m_1} {2m} < r_1^2 < \frac {2 m_1 + m_2} {2m} \quad \text{and} \quad \frac {m_2} {2m} < r_2^2 < \frac {m_1 + 2 m_2} {2 m}.
	\end{equation*}
	
	Moreover, if $M_1$ and $M_2$ are $PMC$, according to Theorem \ref{th:HForPMCLambdaBiharmonicImmersions}, we obtain even more restrictive ranges for the radii $r_1$ and $r_2$, that is
	\begin{align}
		&\frac {m_1 + 2\sqrt{m_1 - 1}} {2m} \leq r_1^2 \leq \frac {2m_1 + m_2 - 2 \sqrt {m_2 - 1}} {2m}, \label{eq:restrictionsRadiusR1}\\
		&\frac {m_2 + 2\sqrt{m_2 - 1}} {2m} \leq r_2^2 \leq \frac {m_1 + 2m_2 - 2 \sqrt {m_1 - 1}} {2m}. \label{eq:restrictionsRadiusR2}
	\end{align}
	Now, from the last equation of \eqref{eq:biharmonicitySubmanifoldTorusR} we infer the following result which also holds in the case of surfaces.
	
	\begin{proposition}\label{th:H1MaxEquivH2Max}
		Let $\varphi_1 : M_1^{m_1} \to \mathbb S^{n_1} (r_1)$ and $\varphi_2 : M_2^{m_2} \to \mathbb S^{n_2} (r_2)$ be two non-minimal immersions such that $M^m = M_1^{m_1} \times M_2^{m_2}$ is proper biharmonic in $\mathbb S^n$, where $n_1 + n_2 = n - 1$. Then 
		\begin{equation*}
			\left |H^{\varphi_1} \right |^2 = \frac 1 {r_1^2} - \frac {\lambda_1} {m_1} \quad \text{if and only if} \quad \left | H^{\varphi_2} \right |^2 = \frac 1 {r_2^2} - \frac {\lambda_2} {m_2}.
		\end{equation*}
		In this case, $\left | H^{\iota \circ \varphi} \right |^2 = 1$.
	\end{proposition}
	\begin{proof}
		In order to prove the last part, we take into account the second fundamental form of $\mathbb S ^{n_1} (r_1) \times \mathbb S^{n_2} (r_2)$ in $\mathbb S^n$ and prove that the tension field of $\iota \circ \varphi$ is given by
		\begin{equation*}\label{eq:normOfTau}
			|\tau (\iota \circ \varphi)|^2 = |\tau(\varphi_1)|^2 + |\tau (\varphi_2)|^2 + r_1^2 r_2^2 \left ( \frac {m_2} {r_2^2} - \frac {m_1} {r_1^2} \right )^2,
		\end{equation*}
		equivalently
		\begin{equation}\label{eq:NormHMareGeneral}
			\left |H^{\iota \circ \varphi}\right |^2 = \frac {m_1^2} {m^2} \left | H^{\varphi_1} \right |^2 + \frac {m_2^2} {m^2} \left | H^{\varphi_2} \right |^2 + \frac {r_1^2 r_2^2} {m^2} \left ( \frac {m_2} {r_2^2} - \frac {m_1} {r_1^2} \right )^2.
		\end{equation}
		Thus, in our case $\left |H^{\iota \circ \varphi} \right |^2 = 1$.
	\end{proof}
	
	\begin{remark}\label{th:RemarkH1MaxIffH2Max}
		In the hypotheses of the above theorem we have that either both $M_1$ and $M_2$ are pseudo-umbilical or none.
	\end{remark}
	\begin{remark}
		We note that the above result holds without any (further) restrictions on $r_1^2$ and $r_2^2$. Now, consider $\varphi_1 : M_1^{m_1} \to \mathbb S^{n_1} (r_1)$, $m_1 > 2$, and $\varphi_2 : M_2^{m_2} \to \mathbb S^{n_2} (r_2)$, $m_2 > 2$ be two $PMC$ and non-pseudo-umbilical immersions such that $M^m = M_1^{m_1} \times M_2^{m_2}$ is proper biharmonic in $\mathbb S^n$, $n_1 + n_2 = n - 1$. If 
		\begin{equation*}
			\left | H^{\varphi_1} \right |^2 = x_{2,1} \quad \text{and} \quad \left | H^{\varphi_2} \right |^2 = x_{2,2}
		\end{equation*}
		then, using the last equation of the system from equation \eqref{eq:biharmonicitySubmanifoldTorusR}, the radius $r_1$ must satisfy the following equation
		\begin{align}\label{eq:EquationFirstRadius}
			& \left ( 1 - r_1^2 \right ) \frac { (m_1 - 2)^2 - 2 m_1 \left ( m_1 - m r_1^2 \right ) + (m_1 - 2) \sqrt {\left ( 2 m r_1^2 - m_1 \right )^2 - 4 (m_1 - 1)}} {2 r_1 ^2}  \\
			- & r_1 ^2 \frac { (m_2 - 2)^2 - 2 m_2 \left ( m r_1^2 - m_1 \right ) + (m_2 - 2) \sqrt { \left ( m_2 + 2 m_1 - 2 m r_1^2 \right )^2 - 4(m_2 - 1)}} {2 \left ( 1 - r_1^2 \right )} \notag \\
			+ & \left ( 1 - 2r_1^2 \right ) \left ( \frac {r_1^2} {1 - r_1^2} m_2^2 - 2 m_1 m_2 + \frac {1 - r_1^2} {r_1^2} m_1^2 \right ) = 0 \notag
		\end{align}
		It seems difficult to solve equation \eqref{eq:EquationFirstRadius} and to find its solutions that belong to the interval $(0, 1)$. As we will see in Theorem \ref{th:MainTheorem}, in order to overcome this problem we will use the fact that, when $\left | H^{\varphi_1} \right |^2 = x_{2,1}$ and $\left | H^{\varphi_2} \right |^2 = x_{2,2}$, the manifolds $M_1$ and $M_2$ split as $M_1 = L_1 \times \mathbb S^1 (b_1)$ and $M_2 = L_2 \times \mathbb S^1 (b_2)$. Then we will use a technique similar to that developed in \cite{BalmusMontaldoOniciuc2009} and \cite{Zhang2011} (see also \cite{WangWu2012}).
	\end{remark}
	
	From Proposition \ref{th:ProperBiharmonicSubmanifoldsInSpheresGeneral} and Proposition \ref{th:H1MaxEquivH2Max}, we can deduce the following result. 
	
	\begin{theorem}\label{th:CharacterizationNormHIs1}
		Let $\varphi_1 : {M_1^{m_1}} \to {\mathbb S ^{n_1}(r_1)}$ and $\varphi_2 : {M_2^{m_2}} \to {\mathbb S^{n_2}(r_2)}$ be two non-minimal immersions such that $M^m = M_1^{m_1} \times M_2^{m_2}$ is proper biharmonic in $\mathbb S^{n}$, $r_1^2 + r_2^2 = 1$ and $n_1 + n_2 = n - 1$. Then $|H^{\iota\circ \varphi}| = 1$ if and only if $\varphi_1$ and $\varphi_2$ are $PMC$ and pseudo-umbilical. Moreover, in this case, we have a codimension reduction result, that is, $M$ is minimal in $\mathbb S^{n-2} (1 / \sqrt 2) \subset \mathbb S^{n-1}$.
	\end{theorem}
	
	\begin{proof}
		It remains to prove only the last part of the theorem. Indeed, from Theorem \ref{th:CMCLambdaBiharmonicImmersion} and up to isometries of $\mathbb R^{n_1 + 1}$ and $\mathbb R^{n_2 + 1}$, we can assume that $M_1$ is minimal in the hypersphere $\mathbb S^{n_1 - 1} (a_1) \times \{ b_1 \}$ of $\mathbb S^{n_1} (r_1)$ and $M_2$ is minimal in the hypersphere $\mathbb S^{n_2 - 1} (a_2) \times \{ b_2 \}$ of $\mathbb S^{n_2} (r_2)$, respectively, where 
		\begin{align*}
			a_i^2 = \frac {m_i} {2m} \quad \text{and} \quad b_i^2 = \frac {-m_i + 2m r_i^2} {2m}, \quad \forall i \in \{1, 2\}.
		\end{align*}
		Rotating $\mathbb S^n$, we can assume that 
		$$
		M = M_1 \times M_2 \subset \mathbb S^{n_1 - 1} (a_1) \times \mathbb S^{n_2 - 1} (a_2) \times \{ (b_1, b_2) \}.
		$$
		Since $b_1^2 + b_2^2 = 1/2$, we can assume that $b_1 = 1 / \sqrt 2$ and $b_2 = 0$ and we end the proof.
	\end{proof}
	
	\begin{corollary}
		Let $\varphi_1 : M_1^{m_1} \to \mathbb S^{m_1 + 1} (r_1)$ and $\varphi_2 : M_2 ^{m_2} \to \mathbb S^{m_2 + 1} (r_2)$ be two non-minimal hypersurfaces such that $M^m = M_1^{m_1} \times M_2^{m_2}$ is proper biharmonic in $\mathbb S^{m+3}$, $r_1^2 + r_2^2 = 1$. Then $\left | H^{\iota \circ \varphi} \right |^2 = 1$ if and only if, up to isometries of $\mathbb S^{m+3}$, $\varphi (M)$ is an open subset of the minimal extrinsic product
		$$
		\mathbb S^{m_1} \left ( \sqrt {\frac {m_1} {2m}} \right ) \times \mathbb S^{m_2} \left ( \sqrt {\frac {m_2} {2m}} \right ) \subset \mathbb S^{m+1} (1 / \sqrt 2) \subset \mathbb S^{m+2}.
		$$
	\end{corollary}
	
	The following result shows that there exists a gap for the range of the mean curvature of $PMC$ proper biharmonic submanifolds $M$ in $\mathbb S^n$, when $M = M_1 \times M_2$, and $M_i$ is $PMC$ and non-minimal in $\mathbb S^{n_i} (r_i)$, $i \in \{ 1, 2\}$. This gap is larger than $((m - 2) / m, 1)$ given by Theorem \ref{th:IntervalHPMCProperBiharmonicImmersions} and, more precisely, it is $((m-4) / m, 1)$.
	
	\begin{theorem}\label{th:MainTheorem}
		Let $\varphi_1 : {M_1^{m_1}} \to {\mathbb S ^{n_1}(r_1)}$ and $\varphi_2 : {M_2^{m_2}} \to {\mathbb S^{n_2}(r_2)}$ be two $PMC$ and non-pseudo-umbilical immersions, $m_1 > 2$, $m_2 > 2$, such that $M^m = M_1^{m_1} \times M_2^{m_2}$ is proper biharmonic in $\mathbb S^{n}$, $r_1^2 + r_2^2 = 1$ and $n_1 + n_2 = n - 1$. Then $|H^{\iota\circ \varphi}| \in \left ( 0, (m-4) / m \right ]$ and $\left |H^{\iota\circ \varphi}\right | = (m-4) / m$ if and only if $r_1^2 = (3m_1 + m_2 - 4) / (4(m-2))$ and locally, up to isometries of $\mathbb S^n$, $\varphi(M)$ is an open subset of the extrinsic product
		$$
		\tilde M_1 \times \mathbb S^1 \left (\frac 1 2 \right ) \times \mathbb S^1 \left ( \frac 1 2 \right ) \subset \mathbb S^{n-4} \left ( \frac 1 {\sqrt 2} \right ) \times \mathbb S^3 \left ( \frac 1 {\sqrt 2} \right ),
		$$
		where $\tilde M_1$ is minimal in $\mathbb S^{n-4} (1 / \sqrt 2)$ and the splitting of $\mathbb R^{n+1} = \mathbb R^{n-3} \times \mathbb R^4$ does not depend on the point of $M$. Moreover, if $M$ is complete, then the previous decomposition of $\varphi (M)$ is global, where $\tilde M_1$ is a complete minimal submanifold in $\mathbb S^{n-4} \left ( 1 / {\sqrt 2} \right )$.
	\end{theorem}
	
	\begin{proof}
		According to Theorem \ref{th:HForPMCLambdaBiharmonicImmersions} we have
		\begin{equation*}
			\left | H^{\varphi_1} \right |^2 \leq x_{2,1} \quad \text{and} \quad \left | H^{\varphi_2} \right |^2 \leq x_{2,2},
		\end{equation*}
		where in our case
		\begin{align*}
			    x_{2,1} =& \frac {(m_1 - 2)^2 - m_1 \lambda_1 r_1^2 + (m_1 - 2) \sqrt {\left ( m_1 - \lambda_1 r_1^2 \right )^2 - 4(m_1 - 1)}} {2 m_1^2 r_1^2},\\
			\lambda_1   =& 2 r_2^2 \left ( \frac {m_1} {r_1^2} - \frac {m_2} {r_2^2} \right ),\\
			    x_{2,2} =& \frac {(m_2 - 2)^2 - m_2 \lambda_2 r_2^2 + (m_2 - 2) \sqrt {\left ( m_2 - \lambda_2 r_2^2 \right )^2 - 4 (m_2 - 1)}} {2 m_2^2 r_2^2},\\
			\lambda_2   =& 2 r_1^2 \left ( \frac {m_2} {r_2^2} - \frac {m_1} {r_1^2} \right ).
		\end{align*}
		Therefore, from \eqref{eq:NormHMareGeneral}, the squared mean curvature of $M$ in $\mathbb S^n$ is bounded as follows
		\begin{align*}
			\left | H^{\iota \circ \varphi} \right |^2 \leq&  \frac {m_1^2} {m^2} x_{2,1} + \frac {m_2^2} {m^2} x_{2,2} + \frac {r_1^2 r_2^2} {m^2} \left ( \frac {m_2} {r_2^2} - \frac {m_1} {r_1^2} \right )^2.
		\end{align*}
		We look for $M$ such that its mean curvature reaches this bound. Then, locally, from Theorem \ref{th:HForPMCLambdaBiharmonicImmersions}, we have
		\begin{equation*}
			M_1 = L_1 \times \mathbb S^1 (b_1) \quad \text{and} \quad M_2 = L_2 \times \mathbb S^1 (b_2),
		\end{equation*}
		where $L_1$ is minimal in $\mathbb S^{n_1 - 2} (a_1)$, $L_2$ is minimal in $\mathbb S^{n_2 - 2} (a_2)$ and
		\begin{equation}\label{eq:radiiComponents}
			\left \{
			\begin{array}{l}
				\displaystyle a_i^2 = \frac {3m_i - 2 - \lambda_i r_i^2 - \sqrt { \left ( m_i - \lambda_i r_i^2 \right )^2 - 4 (m_i - 1)}} {2 \left (2 m_i - \lambda_i r_i^2 \right ) } r_i^2, \\[13pt]
				\displaystyle b_i^2 = \frac {m_i + 2 - \lambda_i r_i^2 + \sqrt { \left ( m_i - \lambda_i r_i^2 \right ) ^2 - 4 (m_i - 1)}} {2 \left ( 2 m_i - \lambda_i r_i^2 \right )} r_i^2.
			\end{array}
			\right ., \forall i \in \{1, 2\}.
		\end{equation}
		As we mentioned in the last remark we are not going to solve equation \eqref{eq:EquationFirstRadius}, but we use a different method. More precisely, we start with 
		$$
		M = L_1 \times \mathbb S^1 (b_1) \times L_2 \times \mathbb S^1 (b_2) \to \mathbb S^n,
		$$
		where $L_1$ is minimal in $\mathbb S^{n_1 - 2} (a_1)$ and $L_2$ is minimal in $\mathbb S^{n_2 - 2} (a_2)$, and $a_1$, $b_1$, $a_2$, $b_2$ satisfy only the conditions $a_i^2 + b_i^2 = r_i^2$, $i \in \{ 1, 2 \}$ and $r_1^2 + r_2^2 = 1$ (not necessarily \eqref{eq:radiiComponents}). From the biharmonicity condition for $M$ in $\mathbb S^n$ we will find all solutions and then we look for that ones that satisfy \eqref{eq:radiiComponents}.
		
		We can check that $M$ is $PMC$ in $\mathbb S^{n}$ and then we recall that $M$ is biharmonic if and only if
		\begin{equation}\label{eq:BiharmonicityB}
			\trace B \left ( \cdot, A _{H^{\iota \circ \varphi}} (\cdot) \right ) = m H^{\iota \circ \varphi}.
		\end{equation}
		In order to compute $H^{\iota \circ \varphi}$, we first denote by
		$$
		\overline r = \left ( x_1^1, \ldots, x_1^{n_1 - 1}, x_2^1, x_2^2, x_3^1, \dots, x_3^{n_2-1}, x_4^1, x_4^2 \right ) \in \mathbb R^{n_1 - 1} \times \mathbb R^2 \times \mathbb R^{n_2-1} \times \mathbb R^2
		$$ 
		the position vector field in $\mathbb R^{n+2}$. Along 
		$$
		M = M_1 \times M_2 = L_1 \times \mathbb S^1 (b_1) \times L_2 \times \mathbb S^1 (b_2),
		$$
		we set
		\begin{equation*}
			\eta_1 = \frac 1 {a_1} x_1, \quad \eta_2 = \frac 1 {b_1} x_2, \quad \eta_3 = \frac 1 {a_2} x_3, \quad \eta_4 = \frac 1 {b_2} x_4.
		\end{equation*}
		It is easy to check that $\eta_1$, $\eta_2$, $\eta_3$ and $\eta_4$ are unit normal vector fields to $\mathbb S^{n_1-2} (a_1)$, $\mathbb S^1 (b_1)$, $\mathbb S^{n_2-2} (a_2)$ and $\mathbb S^1 (b_2)$ in $\mathbb R^{n_1-1}$, $\mathbb R^2$, $\mathbb R^{n_2-1}$, $\mathbb R^2$, respectively. 
		
		Then, we consider the geodesic frame fields around 
		\begin{itemize}
			\item $x_1$ and tangent to $L_1^{m_1-1}$, denoted by $\left \{ X_{1,i} \right \}_{i \in \overline {1, m_1-1}}$;
			\item $x_2$ and tangent to $\mathbb S^1 (b_1)$, denoted by $\left \{ X_{2, 1} \right \}$;
			\item $x_3$ and tangent to $L_2 ^{m_2-1}$, denoted by $\left \{ X_{3,i} \right \} _{i \in \overline {1, m_2-1}}$;
			\item $x_4$ and tangent to $\mathbb S^1 (b_2)$, denoted by $\left \{ X_{4,1} \right \}$.
		\end{itemize}
		The union of these frame fields is a geodesic frame field around $p = (x_1, x_2, x_3, x_4)$ of $M$.
		At $p$ we have
		\begin{align*}
			B(X_{1,i}, X_{1,i}) =& \nabla^{\mathbb S^{n}} _{X_{1,i}} X_{1,i} - \nabla_{X_{1, i}} X_{1,i}\\
								=& \nabla^{\mathbb R^{n+1}} _{X_{1,i}} X_{1,i} + \left \langle X_{1,i}, X_{1,i} \right \rangle \overline r\\
								=& \nabla^{\mathbb R^{n_1-1}} _{X_{1,i}} X_{1,i} + \overline r\\
								=& \nabla^{\mathbb S^{n_1 - 2}} _{X_{1,i}} X_{1,i} - \frac 1 {a_1} \eta_1 + \overline r\\
								=& \nabla^{\mathbb S^{n_1 - 2}} _{X_{1,i}} X_{1,i} - \nabla^{L_1} _{X_{1,i}} X_{1,i} - \frac 1 {a_1} \eta_1 + \overline r\\
								=& B^1 (X_{1, i}, X_{1,i}) - \frac 1 {a_1} \eta_1 + \overline r,
		\end{align*}
		where $B^1$ is the second fundamental form of $L_1$ in $\mathbb S^{n_1-2} (a_1)$. Thus
		\begin{align*}
			\sum _{i=1} ^{m_1-1} B(X_{1,i}, X_{1,i}) =& m_1 H^1 - \frac {m_1 - 1} {a_1} \eta_1 + (m_1 - 1) \overline r\\
												     =& - \frac {m_1 - 1} {a_1} \eta_1 + (m_1 - 1) \overline r.
		\end{align*}
		In order to compute $B(X_{2,1}, X_{2,1})$ we proceed in the same way and obtain
		\begin{equation*}
			B(X_{2,1}, X_{2,1}) = - \frac 1 {b_1} \eta_2 + \overline r.
		\end{equation*}
		Similarly,
		\begin{align*}
			\sum _{i=1} ^{m_2-1} B(X_{3,i}, X_{3,i}) =& - \frac {m_2 - 1} {a_2} \eta_3 + (m_2 - 1) \overline r,\\
								 B(X_{4,1}, X_{4,1}) =& - \frac 1 {b_2} \eta_4 + \overline r.
		\end{align*}
		Therefore, since
		$$
		\overline r = a_1 \eta_1 + b_1 \eta_2 + a_2 \eta_3 + b_2 \eta_4,
		$$
		 the mean curvature vector field of $M$ in $\mathbb S^n$ is given by
		\begin{align}
			m H^{\iota\circ \varphi} =& \sum _{i=1} ^{m_1-1} B(X_{1,i}, X_{1,i}) + B(X_{2,1}, X_{2,1}) + \sum _{i=1} ^{m_2-1} B(X_{3,i}, X_{3,i}) + B(X_{4,1}, X_{4,1}) \label{eq:normHTheMainTheorem} \\
								  =& \left ( - \frac {m_1 - 1} {a_1} + m a_1 \right ) \eta_1 + \left ( - \frac 1 {b_1} + m b_1 \right ) \eta_2 + \left ( - \frac {m_2 - 1} {a_2} + m a_2 \right ) \eta_3 \notag \\
								   & + \left ( - \frac 1 {b_2} + m b_2 \right ) \eta_4. \notag 
		\end{align}
		Next, we compute the shape operator associated to $H^{\iota \circ \varphi}$. We have
		\begin{align*}
			\nabla^{\mathbb S^{n}} _{X_{1,i}} H^{\iota \circ \varphi} =& \nabla^{\mathbb R^{n+1}} _{X_{1,i}} H^{\iota \circ \varphi} + \left \langle X_{1,i}, H^{\iota \circ \varphi} \right \rangle \overline r \\
								=& \frac 1 {m a_1} \left ( - \frac {m_1 - 1} {a_1} + m a_1 \right ) \nabla ^{\mathbb R^{n_1 - 1}} _{X_{1, i}} x_1 \\
								=& \frac 1 m \left ( - \frac {m_1 - 1} {a_1^2} + m \right ) X_{1, i}. 
		\end{align*}
		Similarly,
		\begin{align*}
			\nabla^{\mathbb S^{n}} _{X_{2,1}} H^{\iota \circ \varphi} =& \frac 1 m \left ( - \frac 1 {b_1^2} + m \right ) X_{2,1},\\
			\nabla^{\mathbb S^{n}} _{X_{3,i}} H^{\iota \circ \varphi} =& \frac 1 m \left ( - \frac {m_2 - 1} {a_2^2} + m \right ) X_{3, i},\\
			\nabla^{\mathbb S^{n}} _{X_{4,1}} H^{\iota \circ \varphi} =& \frac 1 m \left ( - \frac 1 {b_2} + m \right ) X_{4,1}.
		\end{align*}
		Thus, for any $i \in \overline {1, m_1 - 1}$ and $j \in \overline {1, m_2 - 1}$, we have
		\begin{equation*}
			\nabla ^{\perp_{\iota\circ \varphi}} _{X_{1, i}} H^{\iota \circ \varphi} = \nabla ^{\perp_{\iota\circ \varphi}} _{X_{2, 1}} H^{\iota \circ \varphi} = \nabla ^{\perp_{\iota\circ \varphi}} _{X_{3, j}} H^{\iota \circ \varphi} = \nabla ^{\perp_{\iota\circ \varphi}} _{X_{4, 1}} H^{\iota \circ \varphi} = 0
		\end{equation*}
		and
		\begin{align*}
			A_{H^{\iota \circ \varphi}} X_{1, i} =& \frac 1 m \left ( \frac {m_1 - 1} {a_1^2} - m \right ) X_{1, i},\\
			A_{H^{\iota \circ \varphi}} X_{2, 1} =& \frac 1 m \left ( \frac 1 {b_1^2} - m \right ) X_{2,1},\\
			A_{H^{\iota \circ \varphi}} X_{3, j} =& \frac 1 m \left ( \frac {m_2 - 1} {a_2^2} - m \right ) X_{3, i},\\
			A_{H^{\iota \circ \varphi}} X_{4, 1} =& \frac 1 m \left ( \frac 1 {b_2} - m \right ) X_{4,1}.
		\end{align*}
		Now we compute the left-hand side of \eqref{eq:BiharmonicityB}. From the expressions of the shape operator and the second fundamental form computed above, we obtain
		\begin{align*}
			\trace B \left ( \cdot, A_{H^{\iota \circ \varphi}} (\cdot) \right ) =& \sum _{i=1} ^{m_1 - 1} B\left ( X_{1,i}, A_{H^{\iota \circ \varphi}} X_{1, i} \right ) + B \left ( X_{2,1}, A_{H^{\iota \circ \varphi}} X_{2,1} \right )\\
			   															   & + \sum _{i=1} ^{m_2 - 1} B \left ( X_{3, i} , A_{H^{\iota \circ \varphi}} X_{3, i} \right ) + B \left ( X_{4,1}, A_{H^{\iota \circ \varphi}} X_{4,1} \right )\\
			   															  =& \frac {m_1 - 1} m \left ( \frac {m_1 - 1} {a_1^2} - m \right ) \left ( - \frac 1 {a_1} \eta_1 + a_1 \eta_1 + b_1 \eta_2 + a_2 \eta_3 + b_2 \eta_4 \right )\\
			   															   & + \frac 1 m \left ( \frac 1 {b_1^2} - m \right ) \left ( a_1 \eta_1 -\frac 1 {b_1} \eta_2 + b_1 \eta_2 + a_2 \eta_3 + b_2 \eta_4 \right )\\
			   															   & + \frac {m_2 - 1} m \left ( \frac {m_2 - 1} {a_2^2} - m \right ) \left ( a_1 \eta_1 + b_1 \eta_2 - \frac 1 {a_2} \eta_3 + a_2 \eta_3 + b_2 \eta_4 \right )\\
			   															   & + \frac 1 m \left ( \frac 1 {b_2^2} - m \right ) \left ( a_1 \eta_1 + b_1 \eta_2 + a_2 \eta_3 - \frac 1 {b_2} \eta_4 + b_2 \eta_4 \right ).
		\end{align*}
		From equation \eqref{eq:BiharmonicityB}, the expression of $H^{\iota \circ \varphi}$ and the fact that $\eta_1$, $\eta_2$, $\eta_3$ and $\eta_4$ are linearly independent we get that
		\begin{equation*}
			\left \{
			\begin{array}{l}
				\displaystyle \frac {a_1^2} {m_1 - 1} \left ( \frac {(m_1 - 1)^2} {a_1^2} + \frac 1 {b_1^2} + \frac {(m_2 - 1)^2} {a_2^2} + \frac 1 {b_2^2} - 2 m^2 \right ) + 2m - \frac {m_1 - 1} {a_1^2} = 0\\[13pt]
				\displaystyle b_1^2 \left ( \frac {(m_1 - 1)^2} {a_1^2} + \frac 1 {b_1^2} + \frac {(m_2-1)^2} {a_2^2} + \frac 1 {b_2^2} - 2 m^2 \right ) + 2m - \frac 1 {b_1^2} = 0\\[13pt]
				\displaystyle \frac {a_2^2} {m_2-1} \left ( \frac {(m_1 - 1)^2} {a_1^2} + \frac 1 {b_1^2} + \frac {(m_2-1)^2} {a_2^2} + \frac 1 {b_2^2} - 2m^2 \right ) + 2m - \frac {m_2-1} {a_2^2} = 0\\[13pt]
				\displaystyle b_2^2 \left ( \frac {(m_1 - 1)^2} {a_1^2} + \frac 1 {b_1^2} + \frac {(m_2 - 1)^2} {a_2^2} + \frac 1 {b_2^2} - 2 m^2 \right ) + 2m - \frac 1 {b_2^2} = 0
			\end{array}
			\right ..
		\end{equation*}
		We set	 
		$$
		\alpha_1 = \frac {a_1^2} {m_1 - 1}, \quad \alpha_2 = b_1^2, \quad \alpha_3 = \frac {a_2^2} {m_2 - 1}, \quad \alpha_4 = b_2^2
		$$
		and 
		$$
		d = \frac {(m_1 - 1)^2} {a_1^2} + \frac 1 {b_1^2} + \frac {(m_2-1)^2} {a_2^2} + \frac 1 {b_2^2}.
		$$
		With these notations the above system becomes 
		\begin{equation*}
			\left \{
			\begin{array}{l}
				\displaystyle \alpha_1 \left ( d - 2m^2 \right ) + 2m - \frac 1 {\alpha_1} = 0\\[10pt]
				\displaystyle \alpha_2 \left ( d - 2m^2 \right ) + 2m - \frac 1 {\alpha_2} = 0\\[10pt]
				\displaystyle \alpha_3 \left ( d - 2m^2 \right ) + 2m - \frac 1 {\alpha_3} = 0\\[10pt]
				\displaystyle \alpha_4 \left ( d - 2m^2 \right ) + 2m - \frac 1 {\alpha_4} = 0
			\end{array}
			\right .
		\end{equation*}
		and thus the biharmonicity condition is equivalent to this algebraic system.
		
		In this way, $\alpha_k$, $k \in \overline {1, 4}$, are the positive solutions of the following second degree equation
		\begin{equation} \label{eq:SecondDegreeEquationMainTheorem}
			\left ( 2m^2 - d \right ) z^2 - 2mz + 1 = 0.
		\end{equation}
		Therefore, $\alpha_k$'s have at most two distinct values.
		
		Since equation \eqref{eq:SecondDegreeEquationMainTheorem} has at most two distinct real solutions, we distinguish two cases:
		
		i) If $\alpha_k = \alpha$, for any $k \in \overline {1, 4}$, then since $a_1^2 + b_1^2 + a_2^2 + b_2^2 = 1$, we get that $\alpha = 1 / m$ which is one of the two solutions of \eqref{eq:SecondDegreeEquationMainTheorem}.
		
		ii) If we have two distinct $\alpha_k$'s, then the discriminant of equation \eqref{eq:SecondDegreeEquationMainTheorem} must be positive, that is $\Delta = 4 \left ( d - m^2 \right ) > 0$. Thus $d > m^2$ and the solutions are
		\begin{equation}\label{eq:SolutionsSecondDegreeEcuationEnd}
			z_1 = \frac {m - \sqrt {d - m^2}} {2m^2 - d} \quad \text{and} \quad z_2 = \frac {m + \sqrt {d - m^2}} {2m^2 - d}. 
		\end{equation}
		We have 
		\begin{align*}
			1 =& a_1^2 + b_1^2 + a_2^2 + b_2^2\\
			  =& (m_1 - 1) \alpha_1 + \alpha_2 + (m_2 - 1) \alpha_3 + \alpha_4\\
			  =& N_1 z_1 + N_2 z_2,
		\end{align*}
		where $N_1, N_2 \in \mathbb N \backslash \{ 0 \}$ and $N_1 + N_2 = m$.
		Substituting $z_1$ and $z_2$ in the last relation we obtain 
		\begin{equation*}
			(N_2 - N_1) \sqrt {d - m^2} = m^2 - d.
		\end{equation*}
		From here we notice that $N_1 > N_2$. 
		
		Denoting $y = \sqrt {d - m^2}$, $y>0$, the last relation becomes 
		$$
		(N_2 - N_1) y + y^2 = 0,
		$$
		which implies that either $y = 0$, that is impossible, or $y = N_1 - N_2$, i.e. $d = (N_1 - N_2)^2 + m^2$. Replacing $d$ in \eqref{eq:SolutionsSecondDegreeEcuationEnd} we obtain
		\begin{equation}\label{eq:relationsBetweenRatiosBiharmonicity}
			z_1 = \frac 1 {2 N_1} \quad \text{and} \quad z_2 = \frac 1 {2 N_2}.
		\end{equation}
		
		From \eqref{eq:relationsBetweenRatiosBiharmonicity} and case i) we obtain that there are only the following $8$ possible cases:
		\begin{enumerate}[label = \arabic*), ref=\arabic*)]
			\item $\alpha_1 = \alpha_2 = \alpha_3 = \alpha_4 = 1 / m$, \label{case:0}
			\item $\alpha_1 = \alpha_2 = z_1$ and $\alpha_3 = \alpha_4 = z_2$, then $N_1 = (m_1 - 1) + 1 = m_1$, $N_2 = m_2$. Thus $\alpha_1 = \alpha_2 = 1 / (2 N_1) = 1 / (2m_1)$ and $\alpha_3 = \alpha_4 = 1 / (2m_2)$, \label{case:1}
			\item $\alpha_1 = \alpha_3 = 1 / (2(m-2))$ and $\alpha_2 = \alpha_4 = 1 / 4$, \label{case:2}
			\item $\alpha_1 = \alpha_4 = 1 / (2m_1)$ and $\alpha_2 = \alpha_3 = 1 / (2m_2)$, \label{case:3}
			\item $\alpha_1 = \alpha_2 = \alpha_3 = 1 / (2(m-1))$ and $\alpha_4 = 1 / 2$, \label{case:4}
			\item $\alpha_1 = \alpha_2 = \alpha_4 = 1 / (2(m_1+1))$ and $\alpha_3 = 1 / (2(m_2-1))$, \label{case:5}
			\item $\alpha_1 = \alpha_3 = \alpha_4 = 1 / (2(m-1))$ and $\alpha_2 = 1 / 2$, \label{case:6}
			\item $\alpha_2 = \alpha_3 = \alpha_4 = 1 / (2(m_2+1))$ and $\alpha_1 = 1 / (2(m_1 - 1)$. \label{case:7}
		\end{enumerate}
		Using the fact that $L_1$ is minimal in $\mathbb S^{n_1 - 2} (a_1)$ and taking into account the second fundamental form of $\mathbb S^{n_1 - 2} (a_1) \times \mathbb S^1 (b_1)$ in $\mathbb S^{n_1}$ we obtain that
		\begin{equation*}
			\left |H^{\varphi_1} \right |^2 = \frac {a_1^2 b_1^2} {m_1^2 r_1^2} \left ( \frac 1 {b_1^2} - \frac {m_1 - 1} {a_1^2} \right )^2 = \frac {a_1^2 b_1^2} {m_1^2 r_1^2} \left ( \frac 1 {\alpha_2} - \frac 1 {\alpha_1} \right )^2.
		\end{equation*}
		Similarly,
		\begin{equation*}
			\left |H^{\varphi_2} \right |^2 = \frac {a_2^2 b_2^2} {m_2^2 r_2^2} \left ( \frac 1 {b_2^2} - \frac {m_2 - 1} {a_2^2} \right )^2 = \frac {a_2^2 b_2^2} {m_2^2 r_2^2} \left ( \frac 1 {\alpha_4} - \frac 1 {\alpha_3} \right )^2.
		\end{equation*}
		Note that if $\alpha_1 = \alpha_2$ or $\alpha_3 = \alpha_4$, then $M_1$ is minimal in $\mathbb S^{n_1} (r_1)$ or $M_2$ is minimal in $\mathbb S^{n_2} (r_2)$, respectively. Since $M_1$ and $M_2$ are not minimal, the above list narrows down to \ref{case:2} and \ref{case:3}.
		
		Now we study separately the remaining cases. From \ref{case:2} we obtain the radii
		\begin{align*}
			a_1^2 = \frac {m_1 - 1} {2(m - 2)}, \quad a_2^2 =& \frac {m_2 - 1} {2(m - 2)}, \quad b_1^2 = b_2^2 = \frac 1 4,\\
			r_1^2 = \frac {3 m_1 + m_2 - 4} {4(m - 2)},& \quad r_2^2 = \frac {m_1 + 3m_2 - 4} {4(m - 2)}.
		\end{align*}
		Note that $r_1$ and $r_2$ satisfy the relations
		\begin{align*}
			&\frac {m_1 + 2\sqrt{m_1 - 1}} {2m} \leq r_1^2 \leq \frac {2m_1 + m_2 - 2 \sqrt {m_2 - 1}} {2m},\\
			&\frac {m_2 + 2\sqrt{m_2 - 1}} {2m} \leq r_2^2 \leq \frac {m_1 + 2m_2 - 2 \sqrt {m_1 - 1}} {2m}.
		\end{align*}
		Thus 
		$$
		\left |H^{\varphi_1} \right |^2 = \frac {2(m_1 - 1) (m - 4)^2} {m_1^2 (3m_1 + m_2 - 4)} \quad \text{and} \quad \left |H^{\varphi_2} \right |^2 = \frac {2(m_2 - 1) (m - 4)^2} {m_2^2 (m_1 + 3m_2 - 4)}.
		$$
		Note that, since $m^2 - 4m_1 - 4m + 8$ and $m^2 - 4m_2 - 4m + 8$ which appear in the expressions of $a_i$ and $x_{2, i}$ are positive for any $m_1 > 2$ and $m_2 > 2$, the radii $a_i$ and $r_i$, $i \in \{1, 2\}$, satisfy equation \eqref{eq:radiiComponents} and 
		$$
		\left |H^{\varphi_1} \right |^2 = x_{2,1} \quad \text{and} \quad \left |H^{\varphi_2} \right |^2 = x_{2,2}.
		$$
		Using equation \eqref{eq:NormHMareGeneral} or \eqref{eq:normHTheMainTheorem} we obtain that 
		\begin{equation*}
			|H^{\iota \circ \varphi}|^2 = \frac {(m-4)^2} {m^2}.
		\end{equation*}
		The submanifold $M$ is $PMC$ and has $A_{H^{\iota \circ \varphi}}$ parallel. Moreover, taking into account the expressions of the shape operators $A^{\varphi_1}_{H^{\varphi_1}}$ and $A^{\varphi_2}_{H^{\varphi_2}}$ or from \eqref{eq:ThetaFormula} and \eqref{eq:OmegaFormula}, we obtain that the distinct eigenvalues of $A^{\varphi_1} _{H^{\varphi_1}}$ are 
		\begin{equation*}
			\theta_1 = \frac {2(m - 4) (m - 2)} {m_1 (3m_1 + m_2 - 4)}, \quad \omega_1 = - \frac {4 (m_1 - 1) (m - 4)} {m_1 (3m_1 + m_2 - 4)},
		\end{equation*}
		with the multiplicities $m_1 - 1$ and $1$, respectively, and the distinct eigenvalues of $A^{\varphi_2} _{H^{\varphi_2}}$ are 
		\begin{equation*}
			\theta_2 = \frac {2(m - 4) (m - 2)} {m_2 (m_1 + 3m_2 - 4)}, \quad \omega_2 = - \frac {4 (m_2 - 1) (m - 4)} {m_2 (m_1 + 3m_2 - 4)},
		\end{equation*}
		with the multiplicities $m_2 - 1$ and $1$, respectively. Therefore, the distinct eigenvalues of $A_{H^{\iota\circ \varphi}}$ are 
		\begin{equation*}
			\mu_1 = \frac {m - 4} m \quad \text{and} \quad \mu_2 = - \frac {m - 4} m,
		\end{equation*}
		with the multiplicities $m - 2$ and $2$, respectively.
		
		We have 
		\begin{equation*}
			L_1^{m_1 - 1} \times L_2^{m_2 - 1} \longrightarrow \mathbb S^{n_1 - 2} (a_1) \times \mathbb S^{n_2 - 2} (a_2) \longrightarrow \mathbb S^{n - 4} \left ( \frac 1 {\sqrt 2} \right ).
		\end{equation*}
		Taking into account the fact that $L_1$ is minimal in $\mathbb S^{n_1 - 2} (a_1)$, $L_2$ is minimal in $\mathbb S^{n_2 - 2} (a_2)$ and the second fundamental form of $\mathbb S^{n_1 - 2} (a_1) \times \mathbb S^{n_2 - 2} (a_2)$ in $\mathbb S^{n - 4} (1 / \sqrt 2)$, we obtain that $L_1 \times L_2$ is minimal in $\mathbb S^{n - 4} (1 / \sqrt 2)$.
		
		This agrees with Proposition \ref{th:Theorem170}.
		
		From \ref{case:3} we obtain the radii
		\begin{align*}
			a_1^2 = \frac {m_1 - 1} {2m_1},  \quad b_1^2 = \frac 1 {2m_2}&, \quad a_2^2 = \frac {m_2 - 1} {2m_2}, \quad b_2^2 = \frac 1 {2m_1},\\
			r_1^2 = \frac {m_1 m_2 - m_2 + m_1} {2m_1m_2},& \quad r_2^2 = \frac {m_1 m_2 - m_1 + m_2} {2m_1m_2}.
		\end{align*}
		Note that $r_1$ and $r_2$ satisfy the relations
		\begin{align*}
			&\frac {m_1 + 2\sqrt{m_1 - 1}} {2m} \leq r_1^2 \leq \frac {2m_1 + m_2 - 2 \sqrt {m_2 - 1}} {2m}, \\
			&\frac {m_2 + 2\sqrt{m_2 - 1}} {2m} \leq r_2^2 \leq \frac {m_1 + 2m_2 - 2 \sqrt {m_1 - 1}} {2m}. 
		\end{align*}
		Thus
		\begin{equation*}
			\left | H^{\varphi_1} \right |^2 = \frac {2 (m_1 - 1) (m_1 - m_2)^2} {m_1^2 (m_1 m_2 - m_2 + m_1)} \quad \text{and} \quad \left |H^{\varphi_2} \right |^2 = \frac {2 (m_2 - 1) (m_1 - m_2)^2} {m_2^2 (m_1 m_2 - m_1 + m_2)}.
		\end{equation*}
		
		As in the previous case, we obtain
		\begin{equation*}
			|H^{\iota \circ \varphi}|^2 = \frac {(m_1 - m_2)^2} {m^2}.
		\end{equation*}
		and, again in accord with Proposition \ref{th:Theorem170}, we have 
		\begin{equation*}
			L_1^{m_1 - 1} \times \mathbb S^1 (b_2) \longrightarrow \mathbb S^{n_1 - 2} (a_1) \times \mathbb S^{1} (b_2) \longrightarrow \mathbb S^{n_1} \left ( \frac 1 {\sqrt 2} \right )
		\end{equation*}
		and
		\begin{equation*}
			L_2^{m_2 - 1} \times \mathbb S^1 (b_1) \longrightarrow \mathbb S^{n_2 - 2} (a_2) \times \mathbb S^{1} (b_1) \longrightarrow \mathbb S^{n_2} \left ( \frac 1 {\sqrt 2} \right ).
		\end{equation*}
		Taking into account the fact that $L_1$ is minimal in $\mathbb S^{n_1 - 2} (a_1)$, $L_2$ is minimal in $\mathbb S^{n_2 - 2} (a_2)$ and the second fundamental forms of $\mathbb S^{n_1 - 2} (a_1) \times \mathbb S^{1} (b_2)$ in $\mathbb S^{n_1} (1 / \sqrt 2)$ and $\mathbb S^{n_2 - 2} (a_2) \times \mathbb S^{1} (b_1)$ in $\mathbb S^{n_2} (1 / \sqrt 2)$ respectively, we obtain that $L_1 \times \mathbb S^1 (b_2)$ is minimal in $\mathbb S^{n_1} (1 / \sqrt 2)$ and $L_2 \times \mathbb S^1 (b_1)$ is minimal in $\mathbb S^{n_2} (1 / \sqrt 2)$.
		
		Note that, in case \ref{case:3}, since $m_1^2 + m_2^2 - m_1 m_2^2$ and $m_1^2 + m_2^2 - m_1^2 m_2$ which appear in the expressions of $a_i$, $x_{1, i}$ and $x_{2, i}$ cannot be simultaneously positive, the radii $a_i$ and $r_i$, $i \in \{1, 2\}$, do not satisfy simultaneously equation \eqref{eq:radiiComponents}. Moreover, we have
		\begin{itemize}
			\item if $m_1^2 + m_2^2 - m_1 m_2^2 > 0$ and $m_1^2 + m_2^2 - m_1^2 m_2 < 0$, then
			\begin{equation*}
				\left | H^{\varphi_1} \right |^2 = x_{2,1} \quad \text{and} \quad \left | H^{\varphi_2} \right |^2 = x_{1,2};
			\end{equation*}
			\item if $m_1^2 + m_2^2 - m_1 m_2^2 < 0$ and $m_1^2 + m_2^2 - m_1^2 m_2 > 0$, then
			\begin{equation*}
				\left | H^{\varphi_1} \right |^2 = x_{1,1} \quad \text{and} \quad \left | H^{\varphi_2} \right |^2 = x_{2,2};
			\end{equation*}
			\item if $m_1^2 + m_2^2 - m_1 m_2^2 < 0$ and $m_1^2 + m_2^2 - m_1^2 m_2 < 0$, then
			\begin{equation*}
				\left | H^{\varphi_1} \right |^2 = x_{1,1} \quad \text{and} \quad \left | H^{\varphi_2} \right |^2 = x_{1,2}.
			\end{equation*}
		\end{itemize}
		
		Now we show that any proper biharmonic submanifold $M = M_1 \times M_2$ has 
		$$
		\left | H^{\iota \circ \varphi} \right |^2 \leq \frac {(m-4)^2} {m^2}.
		$$ 
		Using \eqref{eq:biharmonicitySubmanifoldTorusR} and \eqref{eq:NormHMareGeneral} we can express $\left | H^{\iota \circ \varphi} \right |^2$ in two ways 
		\begin{align*}
			\left | H^{\iota\circ \varphi} \right |^2 = \frac {m_1^2} {m^2} \frac 1 {r_1^2} \left | H^{\varphi_1} \right |^2 + \frac 1 {m^2} \left ( m - \frac {m_1} {r_1^2} \right )^2,\\
			\left | H^{\iota\circ \varphi} \right |^2 = \frac {m_2^2} {m^2} \frac 1 {r_2^2} \left | H^{\varphi_2} \right |^2 + \frac 1 {m^2} \left ( m - \frac {m_2} {r_2^2} \right )^2.
		\end{align*}
		From $\left | H^{\varphi_1} \right |^2 \leq x_{2,1}$, $\left | H^{\varphi_2} \right |^2 \leq x_{2,2}$ and the expressions of $x_{2,1}$ and $x_{2,2}$, we find two upper bounds for the squared mean curvature 
		\begin{align*}
			\left | H^{\iota\circ \varphi} \right |^2 \leq& \frac {(m_1 - 2)^2 - 2 m_1^2} {2m^2} \frac 1 {r_1^4} + \frac {m_1} m \frac 1 {r_1^2} + \frac {m_1 - 2} {2m^2} \frac {\sqrt {\left (2mr_1^2 - m_1 \right )^2 - 4(m_1 - 1)}} {r_1^4} \\
													   & + \frac 1 {m^2} \left ( m - \frac {m_1} {r_1^2} \right )^2,\\
			\left | H^{\iota\circ \varphi} \right |^2 \leq& \frac {(m_2 - 2)^2 - 2 m_2^2} {2m^2} \frac 1 {r_2^4} + \frac {m_2} m \frac 1 {r_2^2} + \frac {m_2 - 2} {2m^2} \frac {\sqrt {\left (2mr_2^2 - m_2 \right )^2 - 4(m_2 - 1)}} {r_2^4}\\
													   & + \frac 1 {m^2} \left ( m - \frac {m_2} {r_2^2} \right )^2.
		\end{align*}
		Let $g : I \to \mathbb R$ and $h: J \to \mathbb R$ given by 
		\begin{align*}
			g(x) =& \frac {(m_1 - 2)^2 - 2 m_1^2} {2m^2} \frac 1 {x^2} + \frac {m_1} m \frac 1 {x} + \frac {m_1 - 2} {2m^2} \frac {\sqrt {\left (2mx - m_1 \right )^2 - 4(m_1 - 1)}} {x^2}\\
				  & + \frac 1 {m^2} \left ( m - \frac {m_1} {x} \right )^2,\\
			h(x) =& \frac {(m_2 - 2)^2 - 2 m_2^2} {2m^2} \frac 1 {x^2} + \frac {m_2} m \frac 1 {x} + \frac {m_2 - 2} {2m^2} \frac {\sqrt {\left (2mx - m_2 \right )^2 - 4(m_2 - 1)}} {x^2} \\
				  & + \frac 1 {m^2} \left ( m - \frac {m_2} {x} \right )^2,
		\end{align*}
		where $I$ and $J$ are defined by \eqref{eq:restrictionsRadiusR1} and \eqref{eq:restrictionsRadiusR2}, respectively.
		
		Studying the derivatives of these functions, we obtain that $g$ and $h$ are increasing functions. Note that 
		\begin{align*}
			\left | H^{\iota\circ \varphi} \right |^2 \leq \tilde g(r_1) = g \left ( r_1^2 \right ) \quad \text{and} \quad \left | H^{\iota\circ \varphi} \right |^2 \leq \tilde h(r_1) = h \left ( 1 - r_1^2 \right ).
		\end{align*}
		From here we find a better upper bound, that is, for any $r_1$ we have
		\begin{equation*} 
			\left | H^{\iota\circ \varphi} \right |^2 \leq \min \left \{ \tilde g (r_1), \tilde h (r_1) \right \}.
		\end{equation*} 
		It is not difficult to see that $\tilde g$ is an increasing function, $\tilde h$ is a decreasing function and 
		$$
		\tilde g \left ( \frac {3m_1 + m_2 - 4} {4(m - 2)} \right ) = \tilde h \left ( \frac {3m_1 + m_2 - 4} {4(m - 2)} \right ) = \frac {(m - 4)^2} {m^2}.
		$$ 
		Therefore, we obtain that $\min \left \{ \tilde g (r_1), \tilde h (r_1) \right \} \leq (m - 4)^2 / m^2$ for any $r_1$, thus $\left | H^{\iota \circ \varphi} \right |^2 \leq (m - 4)^2 / m^2$.
		
		We note that there could be other submanifolds $M = M_1 \times M_2$ such that their squared mean curvature reaches the upper bound $\min \left \{ \tilde g, \tilde h \right \}$. For example, the submanifolds described in case \ref{case:3} have 
		$$
		\left | H^{\iota \circ \varphi} \right |^2 = \frac {(m_1 - m_2)^2} {m^2} = \min \left \{ \tilde g, \tilde h \right \},
		$$ 
		when $\left ( m_1^2 + m_2^2 - m_1 m_2^2 \right ) \left (m_1^2 + m_2^2 - m_1^2 m_2 \right ) < 0$.
	\end{proof}

	\begin{corollary}
		Let $\varphi_1 : {M_1^{m_1}} \to {\mathbb S ^{m_1 + 1}(r_1)}$ and $\varphi_2 : {M_2^{m_2}} \to {\mathbb S^{m_2 + 1}(r_2)}$ be two non-pseudo-umbilical hypersurfaces, $m_1 > 2$, $m_2 > 2$, such that $M^m = M_1^{m_1} \times M_2^{m_2}$ is proper biharmonic in $\mathbb S^{m + 3}$, $r_1^2 + r_2^2 = 1$. 
		Then $\left | H^{\iota \circ \varphi} \right |^2 \in (0, (m - 4)^2 / m^2 ]$ and $\left | H^{\iota \circ \varphi} \right |^2 = (m - 4)^2 / m^2$ if and only if $r_1^2 = (3m_1 + m_2 - 4) / (4(m-2))$ and, up to isometries of $\mathbb S^{m+3}$, $\varphi (M)$ is an open subset of the extrinsic product
		$$
		\mathbb S^{m_1 - 1} \left ( \sqrt {\frac {m_1 - 1}  {2(m-2)}} \right ) \times \mathbb S^{m_2 - 1} \left ( \sqrt {\frac {m_2 - 1}  {2(m-2)}} \right ) \times \mathbb S^1 \left (\frac 1 2 \right ) \times \mathbb S^1 \left (\frac 1 2 \right ).
		$$
	\end{corollary}
	
	Now, we continue our study with the case of surfaces. First we recall that the non-minimal and $PMC$ surfaces in $\mathbb S^n (r)$ were already classified in \cite{Chen1972}, \cite{Yau1974} and they are: minimal surfaces in a small hypersphere of $\mathbb S^n (r)$ or $CMC$ surfaces in $\mathbb S^3 (r') \subset \mathbb S^n (r)$, $0 < r' \leq r$.
	
	We obtain the following rigidity result
	
	\begin{theorem}\label{th:TheoremTwoSurfaces}
		Let $\varphi_1 : M_1^2 \to \mathbb S^{n_1} (r_1)$ and $\varphi_2 : M_2^2 \to \mathbb S^{n_2} (r_2)$ be two non-minimal and $PMC$ surfaces such that $M^4 = M_1^2 \times M_2^2$ is proper-biharmonic in $\mathbb S^n$, $r_1^2 + r_2^2 = 1$ and $n_1 + n_2 = n-1$. Then $M_1$ and $M_2$ are pseudo-umbilical in $\mathbb S^{n_1} (r_1)$ and $\mathbb S^{n_2} (r_2)$, respectively and $\left | H^{\iota \circ \varphi} \right | = 1$.
	\end{theorem}
	\begin{proof}
		Since $M_i$ is $PMC$ and proper $\lambda_i$-biharmonic in $\mathbb S^{n_i} (r_i)$, from the classification of parallel surfaces in Euclidean spheres and Theorem \ref{th:lambdaBiharmonicSubmanifoldsThatLieInHyperspheres}, one can prove that either $M_i$ is pseudo-umbilical in $\mathbb S^{n_i} (r_i)$ and $\left | H^{\varphi_i} \right |^2 = 1 / r_i^2 - \lambda_i / m_i$ or $M_i$ is a product of two circles of different radii which lies in $\mathbb S^3 (r_i) \subset \mathbb S^{n_i} (r_i)$ and $\left | H^{\varphi_i} \right |^2 = x_{2,i}$, $i \in \{ 1, 2\}$. 
		
		Now, using Remark \ref{th:RemarkH1MaxIffH2Max}, both $M_1$ and $M_2$ have to be either pseudo-umbilical or a product of two circles of different radii. In the latter case, according to second case of Proposition \ref{th:ExampleProductSphereCircle} when $m = 2$, both $\lambda_1$ and $\lambda_2$ have to be negative, which is a contradiction, since $\lambda_1 \lambda_2 < 0$.
	\end{proof}
	
	As previously, when $\varphi_1 : M_1^2 \to \mathbb S^{n_1} (r_1)$ and $\varphi_2 : M_2^{m_2} \to \mathbb S^{n_2} (r_2)$, $m_2 > 2$, are two non-minimal and $PMC$ submanifolds such that $M^m = M_1^2 \times M_2^{m_2}$ is proper biharmonic in $\mathbb S^n$, $r_1^2 + r_2^2 = 1$ and $n_1 + n_2 = n - 1$, we have two cases: either both $M_1$ and $M_2$ are pseudo-umbilical or none of them is pseudo-umbilical. The latter case is similar to the case treated in Theorem \ref{th:MainTheorem}, since the surface $M_1$ is a product of two circles.
	
	Summarizing Theorem \ref{th:CharacterizationNormHIs1}, Theorem \ref{th:MainTheorem}, Theorem \ref{th:TheoremTwoSurfaces} and the above remark we can state.
	
	\begin{theorem}\label{th:SummarizeTheorem}
		Let $\varphi_1 : M_1 ^{m_1} \to \mathbb S^{n_1} (r_1)$, $m_1 \geq 2$, and $\varphi_2 : M_2^{m_2} \to \mathbb S^{n_2} (r_2)$, $m_2 \geq 2$, be two non-minimal and $PMC$ submanifolds such that $M^m = M_1^{m_1} \times M_2^{m_2}$ is proper biharmonic in $\mathbb S^n$, $r_1 ^2 + r_2^2 = 1$ and $n_1 + n_2 = n - 1$.
		\begin{enumerate}
			\item If $m_1 > 2$ or $m_2 > 2$, then
			$$
			\left | H^{\iota \circ \varphi} \right | \in \left ( 0, \frac {m-4} m \right ] \cup \{ 1 \}.
			$$
			\item If $m_1 = m_2 = 2$, then $\varphi_1$ and $\varphi_2$ are pseudo-umbilical, that is
			$$
			\left | H^{\iota \circ \varphi} \right | = 1.
			$$
		\end{enumerate}
	\end{theorem}
	
	We end the paper proposing an open problem that raised from the following. We can prove that, when $M^m = L_1^{m_1-k_1} \times \mathbb S^{k_1} (b_1) \times L_2^{m_2-k_2} \times \mathbb S^{k_2} (b_2)$ is proper biharmonic in $\mathbb S^n$, where $L_1$ is minimal in $\mathbb S^{n_1 - k_1 - 1} (a_1)$ and $L_2$ is minimal in $\mathbb S^{n_2 - k_2 - 1} (a_2)$, we find only two cases and only two admissible values for $r_1$. In both cases $\left | H^{\iota\circ \varphi} \right | \leq (m-4) / m$. However, it is not clear if for any $r_1$ such that 
	$$
	r_1 ^2 \in \left [ \frac {m_1 + \sqrt {m_1 - 1}} {2m}, \frac {2m_1 + m_2 - \sqrt {m_2 - 1}} {2m} \right ]
	$$
	there exists $M$ as in the hypotheses of Theorem \ref{th:MainTheorem}. So we can state the following problem

	\textbf{Open Problem.} Prove that for any $r_1$ such that
	$$
	r_1 ^2 \in \left [ \frac {m_1 + \sqrt {m_1 - 1}} {2m}, \frac {2m_1 + m_2 - \sqrt {m_2 - 1}} {2m} \right ]
	$$
	there exists $M$ as in the hypotheses of Theorem \ref{th:MainTheorem}.	
	
	\medskip \medskip
	\textbf{Acknowledgements.} We are grateful to Cezar Oniciuc for the useful discussions and suggestions.
	
	\bibliographystyle{abbrv}
	\bibliography{Bibliography.bib}

\begin{thebibliography}{10}

\bibitem{AkutagawaMaeta2013}
K.~Akutagawa and S.~Maeta.
\newblock Biharmonic properly immersed submanifolds in {E}uclidean spaces.
\newblock {\em Geom. Dedicata}, 164:351--355, 2013.

\bibitem{AlencarDoCarmo1994}
H.~Alencar and M.~do~Carmo.
\newblock Hypersurfaces with constant mean curvature in spheres.
\newblock {\em Proc. Amer. Math. Soc.}, 120(4):1223--1229, 1994.

\bibitem{AliasGarcia2010}
L.~J. Al\'{\i}as and S.~C. Garc\'{\i}a-Mart\'{\i}nez.
\newblock On the scalar curvature of constant mean curvature hypersurfaces in
  space forms.
\newblock {\em J. Math. Anal. Appl.}, 363(2):579--587, 2010.

\bibitem{ArroyoGarayMencia1998}
J.~Arroyo, O.~J. Garay, and J.~J. Menc\'{\i}a.
\newblock On a family of surfaces of revolution of finite {C}hen-type.
\newblock {\em Kodai Math. J.}, 21(1):73--80, 1998.

\bibitem{ArvanitoyeorgosKaimakamisMagid2009}
A.~Arvanitoyeorgos, G.~Kaimakamis, and M.~Magid.
\newblock Lorentz hypersurfaces in {$E^4_1$} satisfying {$\Delta\overrightarrow
  H=\alpha\overrightarrow H$}.
\newblock {\em Illinois J. Math.}, 53(2):581--590, 2009.

\bibitem{BalmusMontaldoOniciuc2008}
A.~Balmu\c{s}, S.~Montaldo, and C.~Oniciuc.
\newblock Classification results for biharmonic submanifolds in spheres.
\newblock {\em Israel J. Math.}, 168:201--220, 2008.

\bibitem{BalmusMontaldoOniciuc2009}
A.~Balmu\c{s}, S.~Montaldo, and C.~Oniciuc.
\newblock Classification results and new examples of proper biharmonic
  submanifolds in spheres.
\newblock {\em Note Mat.}, 28:49--61, 2009.

\bibitem{BalmusOniciuc2012}
A.~Balmu\c{s} and C.~Oniciuc.
\newblock Biharmonic submanifolds with parallel mean curvature vector field in
  spheres.
\newblock {\em J. Math. Anal. Appl.}, 386(2):619--630, 2012.

\bibitem{Branding2020}
V.~Branding.
\newblock On interpolating sesqui-harmonic maps between {R}iemannian manifolds.
\newblock {\em J. Geom. Anal.}, 30(1):248--273, 2020.

\bibitem{CaddeoMontaldoOniciuc2002}
R.~Caddeo, S.~Montaldo, and C.~Oniciuc.
\newblock Biharmonic submanifolds in spheres.
\newblock {\em Israel J. Math.}, 130:109--123, 2002.

\bibitem{Chen1972}
B.-Y. Chen.
\newblock On the surface with parallel mean curvature vector.
\newblock {\em Indiana Univ. Math. J.}, 22:655--666, 1972/73.

\bibitem{ChenBook1973}
B.-Y. Chen.
\newblock {\em Geometry of submanifolds}, volume No. 22 of {\em Pure and
  Applied Mathematics}.
\newblock Marcel Dekker, Inc., New York, 1973.

\bibitem{Chen1988}
B.-Y. Chen.
\newblock Null {$2$}-type surfaces in {$E^3$} are circular cylinders.
\newblock {\em Kodai Math. J.}, 11(2):295--299, 1988.

\bibitem{Chen1991}
B.-Y. Chen.
\newblock Some open problems and conjectures on submanifolds of finite type.
\newblock {\em Soochow J. Math.}, 17(2):169--188, 1991.

\bibitem{ChenBook2015}
B.-Y. Chen.
\newblock {\em Total mean curvature and submanifolds of finite type}, volume~27
  of {\em Series in Pure Mathematics}.
\newblock World Scientific Publishing Co. Pte. Ltd., Hackensack, NJ, second
  edition, 2015.

\bibitem{ChenYano1972}
B.~Y. Chen and K.~Yano.
\newblock Pseudo-umbilical submanifolds in a {R}iemannian manifold of constant
  curvature.
\newblock In {\em Differential geometry (in honor of {K}entaro {Y}ano)}, pages
  61--71. Kinokuniya Book Store, Tokyo, 1972.

\bibitem{DajczerTojeiroBook2019}
M.~Dajczer and R.~Tojeiro.
\newblock {\em Submanifold theory}.
\newblock Universitext. Springer, New York, 2019.

\bibitem{Dimitric1992}
I.~Dimitri\'{c}.
\newblock Submanifolds of {$E^m$} with harmonic mean curvature vector.
\newblock {\em Bull. Inst. Math. Acad. Sinica}, 20(1):53--65, 1992.

\bibitem{EellsLemaire1983}
J.~Eells and L.~Lemaire.
\newblock {\em Selected topics in harmonic maps}, volume~50 of {\em CBMS
  Regional Conference Series in Mathematics}.
\newblock Conference Board of the Mathematical Sciences, Washington, DC; by the
  American Mathematical Society, Providence, RI, 1983.

\bibitem{EellsSampson1964}
J.~Eells and J.~H. Sampson.
\newblock Harmonic mappings of {R}iemannian manifolds.
\newblock {\em Amer. J. Math.}, 86:109--160, 1964.

\bibitem{FerrandezLucas1991}
A.~Ferr\'{a}ndez and P.~Lucas.
\newblock Null finite type hypersurfaces in space forms.
\newblock {\em Kodai Math. J.}, 14(3):406--419, 1991.

\bibitem{FetcuOniciuc2022}
D.~Fetcu and C.~Oniciuc.
\newblock Biharmonic and biconservative hypersurfaces in space forms.
\newblock In {\em Differential geometry and global analysis---in honor of
  {T}adashi {N}agano}, volume 777 of {\em Contemp. Math.}, pages 65--90. Amer.
  Math. Soc., [Providence], RI, 2022.

\bibitem{FuYangZhang2022}
Y.~Fu, D.~Yang, and X.~Zhan.
\newblock Recent progress of biharmonic hypersurfaces in space forms {I}n
  {H}onor of {T}adashi {N}agano.
\newblock In {\em Differential geometry and global analysis---in honor of
  {T}adashi {N}agano}, volume 777 of {\em Contemp. Math.}, pages 91--101. Amer.
  Math. Soc., [Providence], RI, 2022.

\bibitem{Garay1994}
O.~J. Garay.
\newblock A classification of certain {$3$}-dimensional conformally flat
  {E}uclidean hypersurfaces.
\newblock {\em Pacific J. Math.}, 162(1):13--25, 1994.

\bibitem{GuanLiVrancken2021}
Z.~Guan, H.~Li, and L.~Vrancken.
\newblock Four dimensional biharmonic hypersurfaces in nonzero space forms have
  constant mean curvature.
\newblock {\em J. Geom. Phys.}, 160:Paper No. 103984, 15, 2021.

\bibitem{HasanisVlachos1995}
T.~Hasanis and T.~Vlachos.
\newblock Hypersurfaces in {$E^4$} with harmonic mean curvature vector field.
\newblock {\em Math. Nachr.}, 172:145--169, 1995.

\bibitem{Jiang1986-1}
G.~Y. Jiang.
\newblock {$2$}-harmonic isometric immersions between {R}iemannian manifolds.
\newblock {\em Chinese Ann. Math. Ser. A}, 7(2):130--144, 1986.

\bibitem{Jiang1986-2}
G.~Y. Jiang.
\newblock {$2$}-harmonic maps and their first and second variational formulas.
\newblock {\em Chinese Ann. Math. Ser. A}, 7(4):389--402, 1986.

\bibitem{Maeta2014-1}
S.~Maeta.
\newblock Biharmonic maps from a complete {R}iemannian manifold into a
  non-positively curved manifold.
\newblock {\em Ann. Global Anal. Geom.}, 46(1):75--85, 2014.

\bibitem{Maeta2014-2}
S.~Maeta.
\newblock Properly immersed submanifolds in complete {R}iemannian manifolds.
\newblock {\em Adv. Math.}, 253:139--151, 2014.

\bibitem{MontaldoOniciucRatto2016}
S.~Montaldo, C.~Oniciuc, and A.~Ratto.
\newblock On cohomogeneity one biharmonic hypersurfaces into the {E}uclidean
  space.
\newblock {\em J. Geom. Phys.}, 106:305--313, 2016.

\bibitem{Nistor2023}
S.~Nistor.
\newblock A new gap for {$CMC$} biharmonic hypersurfaces in {E}uclidean
  spheres.
\newblock {\em J. Math. Anal. Appl.}, 523(2):Paper No. 127030, 14, 2023.

\bibitem{Okumura1974}
M.~Okumura.
\newblock Hypersurfaces and a pinching problem on the second fundamental
  tensor.
\newblock {\em Amer. J. Math.}, 96:207--213, 1974.

\bibitem{Oniciuc2002}
C.~Oniciuc.
\newblock Biharmonic maps between {R}iemannian manifolds.
\newblock {\em An. \c{S}tiin\c{t}. Univ. Al. I. Cuza Ia\c{s}i. Mat. (N.S.)},
  48(2):237--248, 2002.

\bibitem{OniciucPHD}
C.~Oniciuc.
\newblock Tangency and harmonicity properties. {PHD T}hesis.
\newblock {\em Geometry Balkan Press, Bucharest}, 2003.
\newblock http://www.mathem.pub.ro/dgds/mono/dgdsmono.htm.

\bibitem{Oniciuc2012HabilitationThesis}
C.~Oniciuc.
\newblock Biharmonic submanifolds in space forms. {H}abilitation {T}hesis.
\newblock DOI: 10.13140/2.1.4980.5605, www.researchgate.net, 2012.

\bibitem{ChenOu2020}
Y.-L. Ou and B.-Y. Chen.
\newblock {\em Biharmonic submanifolds and biharmonic maps in {R}iemannian
  geometry}.
\newblock World Scientific Publishing Co. Pte. Ltd., Hackensack, NJ, 2020.

\bibitem{OuTang2012}
Y.-L. Ou and L.~Tang.
\newblock On the generalized {C}hen's conjecture on biharmonic submanifolds.
\newblock {\em Michigan Math. J.}, 61(3):531--542, 2012.

\bibitem{WangWu2012}
X.~F. Wang and L.~Wu.
\newblock Proper biharmonic submanifolds in a sphere.
\newblock {\em Acta Math. Sin. (Engl. Ser.)}, 28(1):205--218, 2012.

\bibitem{Yau1974}
S.~T. Yau.
\newblock Submanifolds with constant mean curvature {I}.
\newblock {\em Amer. J. Math.}, 96:346--366, 1974.

\bibitem{Zhang2011}
W.~Zhang.
\newblock New examples of biharmonic submanifolds in {$\Bbb C P^n$} and {$\Bbb
  S^{2n+1}$}.
\newblock {\em An. \c{S}tiin\c{t}. Univ. Al. I. Cuza Ia\c{s}i. Mat. (N.S.)},
  57(1):207--218, 2011.

\end{thebibliography}
\end{document}